\documentclass{amsart}

\usepackage{hyperref}
\newtheorem{theorem}{Theorem}[section]
\newtheorem{lemma}[theorem]{Lemma}
\newtheorem{corollary}[theorem]{Corollary}
\theoremstyle{remark}
\newtheorem{remark}[theorem]{Remark}

\numberwithin{equation}{section}

\newcommand{\dd}{ { \mathrm{d}} }
\newcommand{\ee}{ { \mathrm{e}} }
\newcommand{\cL}{{ \mathcal L}}
\newcommand{\bm}[1]{\begin{bmatrix} #1\end{bmatrix}}

\newcommand{\cU}{\mathcal{U}}
\newcommand{\cI}{\mathcal{I}}

\newcommand{\cH}{{ \mathcal H}}

\newcommand{\cB}{{ \mathcal B}}
\newcommand{\cF}{{ \mathcal F}}
\newcommand{\cD}{{ \mathcal D}}

\newcommand{\cO}{{ \mathcal O}}

\newcommand{\cS}{{ \mathcal S}}
\newcommand{\Cbb}{C_{\mathrm{b}}^2}

\newcommand{\HH}{\cH}
\newcommand{\VV}{\mathcal{V}}
\newcommand{\HHnorm}[1]{\|#1\|}
\newcommand{\HHinner}[2]{\langle#1,#2\rangle}
\newcommand{\dHnorm}[2]{\|#1\|_{\dH^{#2}}}
\newcommand{\dHinner}[3]{\langle#1,#2\rangle_{\dH^{#3}}}
\newcommand{\Lnorm}[1]{\| #1 \|_{L_2(\cD)}}
\newcommand{\Linner}[2]{\langle#1,#2\rangle_{L_2(\cD)}}
\newcommand{\R}{{ \mathbb R}}

\DeclareMathOperator{\Tr}{Tr}
\DeclareMathOperator{\Bor}{Bor}
\newcommand{\HS}{\mathrm{HS}}
\newcommand{\norm}[2]{\|#2\|_{#1}}

\newcommand{\T}[1]{\tilde{#1}}

\newcommand{\D}{D}

\newcommand{\IR}{\mathbb R}
\newcommand{\dH}{\dot{H}}
\newcommand{\bE}{{ \mathbf E}}

\newcommand{\Hnorm}[2]{\| #1\|_{\HH^{#2}}}

\newcommand{\la}{\langle}
\newcommand{\ra}{\rangle}
\newcommand{\EE}{\bE}

\title[Weak convergence of fully discrete approximations] {Weak convergence of finite element approximations of linear stochastic evolution equations with additive noise II. Fully discrete schemes.}

\author[M.~Kov\'acs]{Mih\'aly Kov\'acs}

\address{Department of Mathematics and Statistics,
  University of Otago, P.O.~Box 56, Dune\-din, New Zealand}

\email{mkovacs@maths.otago.ac.nz}

\author[S.~Larsson]{Stig Larsson$^{1,2}$}

\address{
  Department of Mathematical Sciences,
  Chalmers University of Technology and University of Gothenburg,
  SE--412 96 G\"oteborg,
  Sweden}

\email{stig@chalmers.se}

\author[F.~Lindgren]{Fredrik Lindgren$^1$}

\address{
  Department of Mathematical Sciences,
  Chalmers University of Technology and University of Gothenburg,
  SE--412 96 G\"oteborg,
  Sweden}

\email{fredrik.lindgren@chalmers.se}

\thanks{$^1$Supported by the Swedish Research Council (VR)}

\thanks{$^2$Supported by the Swedish Foundation for Strategic Research (SSF)
through GMMC, the Gothenburg Mathematical Modelling Centre.}



\begin{document}

\date{2012--09--18}

\begin{abstract} We present an abstract framework for analyzing the weak error of
  fully discrete approximation schemes for linear evolution equations
  driven by additive Gaussian noise. First, an abstract representation
  formula is derived for sufficiently smooth test functions. The
  formula is then applied to the wave equation, where the spatial
  approximation is done via the standard continuous finite element
  method and the time discretization via an $I$-stable rational
  approximation to the exponential function. It is found that the rate
  of weak convergence is twice that of strong
  convergence. Furthermore, in contrast to the parabolic case, higher
  order schemes in time, such as the Crank-Nicolson scheme, are
  worthwhile to use if the solution is not very regular. Finally we
  apply the theory to parabolic equations and detail a weak error
  estimate for the linearized Cahn-Hilliard-Cook equation as well as
  comment on the stochastic heat equation.
\end{abstract}

\maketitle

\section{Introduction}

Let $\cU,\cH$ be real separable Hilbert spaces and consider the
following abstract stochastic Cauchy problem
\begin{equation}\label{spde}
\dd X(t)+AX(t)\,\dd t=B\,\dd W(t),\, t>0; \quad X(0)=X_0,
\end{equation}
where $-A$ is the generator of a strongly continuous semigroup
$\{E(t)\}_{t\ge 0}$ on $\cH$, $B\in \mathcal{B}(\cU,\cH)$, where
$\mathcal{B}(\cU,\cH)$ denotes the space of bounded linear operators
from $\cU$ to $\cH$. The process $\{W(t)\}_{t\ge 0}$ is a $\cU$-valued Wiener
process with covariance operator $Q\ge 0$ (selfadjoint, positive
semidefinite) with respect to a filtration $\{\mathcal{F}_t\}_{t\ge
  0}$ on a probability space $(\Omega, \mathcal{F},\mathbf{P})$.
 We note that, strictly speaking, the process $\{W(t)\}_{t\ge 0}$ is $\cU$-valued if and only if $\Tr Q<\infty$.
  Finally, $X_0$ is an $\cF_0$-measurable $\cH$-valued random
variable with finite mean. If
\begin{equation}\label{qt}
  \Tr\Big(\int_0^T E(t)BQB^*E(t)^*\,\dd t\Big)<\infty,
\end{equation}
then the unique weak solution is given by (see \cite[Chapter
5]{MR1207136})
\begin{align}  \label{eq:X}
X(t)=E(t)X_0+\int_0^tE(t-s)B\,\dd W(s).
\end{align}
With $G\colon\cH\rightarrow \R$ being a twice Fr\'echet differentiable
function with bounded and continuous first and second derivatives, we
study the weak error
\begin{equation}\label{eq:weakErr}
e(T)=\bE\big(G(\tilde{X}(T))-G(X(T))\big),
\end{equation}
where $\T{X}(T)$ is some approximation of the process $X$ at time $T$.

This paper is a sequel to \cite{KLLweak}.  It consists of three parts. In Section \ref{pre} we define some
central concepts and state important background results used
throughout the paper. In Section \ref{errep} we show that the error
formula in \cite[Theorem 3.1]{KLLweak} holds for a much wider class of
approximations of the solution to \eqref{spde} than stated in that
paper, which is concerned with spatial semidiscretization by finite
elements. Finally, in Sections \ref{sec:hyper} and \ref{sec:para}, the
usefulness of the general error formula in Theorem \ref{main} is
demonstrated through the fact that it can be applied to analyze fully
discrete schemes for a wide class of stochastic evolution equations:
hyperbolic and parabolic alike. The basic line of proof of the general
formula is adapted from \cite{debusscheprintems} which is concerned
with the stochastic heat equation.

The statement of Theorem \ref{main} deserves a motivation. Consider,
for example, the case when $\T{X}(T)$ is the value of a
semidiscretization in space with finite elements, so that
\begin{equation*}\label{eq:pertSol}
\T{X}(T)=\T{E}(T)\T{X}_0+\int_0^T\T{E}(T-s)\T{B}\,\dd W(s)
\end{equation*}
has the same form as the solution \eqref{eq:X} of the original problem
\eqref{spde}. Then an error formula may be derived with the aid of
Kolmogorov's backward equation and It\^o's formula as in
\cite{KLLweak}. It turns out that the error analysis is substantially
simplified if new processes are constructed by multiplying $X(t)$ and
$\T{X}(t)$ by suitable integrating factors. That is, define new,
drift-free processes
\begin{equation}\label{eq:Y}
Y(t)=E(T-t)X(t)=E(T)X_0+\int_0^tE(T-s)B\,\dd W(s)
\end{equation}
and
\begin{equation}\label{eq:Ytilde1}
\T{Y}(t)=\T{E}(T-t)\T{X}(t)=\T{E}(T)\T{X}_0+\int_0^t\T{E}(T-s)\T{B}\,\dd W(s),
\end{equation}
where
\begin{align}
X(T)=Y(T),\quad
\T{X}(T)=\T{Y}(T),\label{eq:dfpEqP}
\end{align}
with $\{Y(t)\}_{t\geq 0}$ being the solution of the equation
\begin{align}\label{eq:eqY}
\dd Y(t)=E(T-t)B\,\dd W(t),\, t>0; \quad Y(0)=E(T)X_0
\end{align}
and $\{\T{Y}(t)\}_{t\geq 0}$ solves
\begin{align}\label{eq:eqYtilde}
\dd \T{Y}(t)=\T{E}(T-t)\T{B}\,\dd W(t),\, t>0; \quad \T{Y}(0)=\T{E}(T)\T{X}_0.
\end{align}
However, if $\tilde{X}(T)$ is the result of a time-stepping scheme, then a process of the form \eqref{eq:X} is not immediate. We note that in \cite{DJR} the interpolation of the time-stepping operator family is performed in a manner that results in a family of deterministic operators $\{\tilde{E}(t)\}_{t\ge 0}$ with a weaker type of semigroup property. This property is sufficient to mimic the computations in \eqref{eq:Ytilde1}, but the operator family does not naturally admit deterministic error estimates. Thus, we have chosen an alternative path, as in \cite{debusscheprintems}, where the discrete semigroup property of the time-stepping operator family is used first, followed by piecewise constant interpolation between the grid points. This will, as we shall see, yield a drift-free process as in the right hand side of \eqref{eq:Ytilde1} such that \eqref{eq:dfpEqP} still holds and known deterministic error estimates can be used almost immediately.
%However, if $\T{X}(T)$ is the result of a time-stepping scheme, then a
%process of the form \eqref{eq:X} is not immediately available, nor is
%the rewriting in \eqref{eq:Ytilde1} possible to perform. On the other
%hand, the construction of a continuous drift-free process as in the
%right hand side of \eqref{eq:Ytilde1} such that \eqref{eq:dfpEqP}
%holds is still possible.
In this case $\{\T{E}(t)\}_{t\geq 0}$ will
not be continuous in $t$ and does not have the semigroup property but
it turns out that these are not necessary. All we need to obtain the
fundamental formulas \eqref{e1}--\eqref{eq:F2} for the error is to
assume that there exists a well-defined process $\{\T{Y}(t)\}_{t\in [0,T]}$ of the form
\begin{equation}\label{eq:tildeY2}
\T{Y}(t)=\T{E}(T)\T{X}_0+\int_0^t\T{E}(T-s)\T{B}\,\dd W(s)
\end{equation}
such that \eqref{eq:dfpEqP} holds. Here $\{\tilde{E}(t)\}_{t\in
  [0,T]}\subset \cB(\cS,\cS)$ and $\tilde{B}\in \mathcal{B}(\cU,\cS)$,
where $\cS$ is a Hilbert subspace of $\cH$ with the same norm
(typically $\cS=\cH$ or $\cS$ is a finite-dimensional subspace of
$\cH$).  The process $\T{Y}$ is then well defined if
\begin{equation}\label{qtdiscrete}
  \Tr\Big(\int_0^T \T{E}(t)\T{B}Q[\T{E}(t)\T{B}]^*\,\dd t\Big)<\infty,
\end{equation}
and in this case it will be the unique weak solution of
\eqref{eq:eqYtilde}.  Here the adjoint
$[\tilde{E}(t)\tilde{B}]^*=\tilde{B}^*\tilde{E}(t)^*\colon \cS\to\cU$
is taken with respect to the scalar product in $\cH$.

In Section \ref{sec:hyper}, we apply the general formula from Theorem
\ref{main} to the error analysis of semi- and fully discrete numerical
schemes for the stochastic wave equation
\begin{equation}\label{eq:wave}
\dd\dot{U}(t)-\Delta U(t)\,\dd t=\dd W(t),
\ U(t)|_{\partial\cD}=0, ~ t> 0;\quad U(0)=U_0,~\dot{U}(0)=V_0,
\end{equation}
where the solution process $\{U(t)\}_{t\ge0}$ and the Wiener process
$\{W(t)\}_{t\ge0}$ take values in $\cU=L_2(\cD)$, where $\cD$ is a
sufficiently nice open bounded domain in $\R^{d}$. Writing
$X(t)=\bm{X_1(t), X_2(t)}^T:=\bm{U(t), \dot{U}(t)}^T$,
$X_0=\bm{X_{0,1},X_{0,2}}^T:=\bm{U_0 , V_0}^T$ and $\Lambda:=-\Delta$,
the wave equation \eqref{eq:wave} can be written in the form
\eqref{spde} with
\begin{equation*}
  A:=\begin{bmatrix}0&-I\\\Lambda&0\end{bmatrix}, \quad
  B:= \begin{bmatrix} 0\\ I \end{bmatrix}.
\end{equation*}
It is well known that $-A$ is the generator of a unitary, strongly
continuous semigroup (and thus a group) on the space
$\cH=L_2(\cD)\times (H_0^1(\cD))^*$.

The first result in Subsection \ref{wave} is a bound of the error $e(T)$
in the more specific context of the wave equation but keeping the
approximating process general. The bound is expressed in terms of the
operators and initial data in \eqref{eq:tildeY2} and \eqref{eq:X}.  In
Subsection \ref{sec:semitime} we apply this to single step rational
approximations of \eqref{spde}, that is, to solutions of the scheme
\begin{equation*}
X^j=R(kA)\big(X^{j-1}+B(W(t_j)-W(t_{j-1}))\big),
\end{equation*}
where $k$ is the step size and where the rational function $R$ fulfills
the approximation and stability properties
\begin{equation}\label{eq:ratApp}
\begin{aligned}
|R(iy)-e^{-iy}|&\leq C |y|^{p+1},&|y|\leq b,\\
|R(i y)|&\leq 1,& y\in \mathbb{R},
\end{aligned}
\end{equation}
for some positive integer $p$ and some $b>0$. If the initial value is smooth enough and
\begin{equation}\label{eq:Qassump}
\norm{\Tr}{\Lambda^{\beta-1/2}Q\Lambda^{-1/2}}<\infty,
\end{equation}
where $\norm{\Tr}{\cdot}$ denotes the trace norm, then for the weak
error in \eqref{eq:weakErr} we have
\begin{equation*}
|e(T)|= O(k^{\min(\tfrac{p}{p+1}2\beta,1)})\quad \text{as }
k\rightarrow 0.
\end{equation*}
It is important to note that, in contrast to the stochastic heat
equation, the convergence rate is improved with higher order schemes
if the noise is sufficiently irregular because $\tfrac{p}{p+1}$
increases with the order, $p$, of the method. This feature appears for
fully discrete schemes investigated in Subsection \ref{timespace} as
well, where the spatial approximation is performed with standard,
continuous, piecewise polynomial finite elements of order $r$ and the
temporal approximation again with rational functions as in
\eqref{eq:ratApp}. If $h$ denotes the size of the finite element mesh
and if \eqref{eq:Qassump} holds, then for the first component $X_1=U$
we have
\begin{equation*}
\big|\bE\big(G(\T{X}_1(T))-G(X_1(T))\big)\big|
=O(k^{\min(\tfrac{p}{p+1}2\beta,1)})+O(h^{\min(\tfrac{r}{r+1}2\beta,r)})
\ \text{as } h,k\to0.
\end{equation*}

It is a general phenomenon that for non-smooth noise, the rate of weak
convergence is twice that of the strong (mean-square) convergence. To
be able to compare the weak rate with the strong rate of fully
discrete schemes for the wave equation, and also because strong results
for such schemes are absent in the literature, we prove a strong
convergence result in Section \ref{strong} for the same
algorithm. Indeed, we find that the strong rate is half the weak rate
also in the case of the wave equation.

Our main motivation in applying the results of Section \ref{errep} to
the stochastic wave equation is the fact that this is one of the
canonical SPDE's that is less understood from the numerical point of
view. Results on weak convergence can be found in \cite{Hwwave} but
are confined to a leap-frog scheme in one spatial variable. Also, the
test-functions in that paper differ from ours, making the results
difficult to compare. In \cite{KLLweak} the stochastic wave equation
is discretized with finite elements in space only but in several
spatial dimensions. The findings are in accordance with the results of
the present paper.

For the stochastic wave equation on the one-dimensional real line with
white noise the strong rate $1/2$ for the leap-frog scheme is
proved in \cite{Walsh}. This rate is also proved to be optimal. For
white noise in one dimension the algorithm discussed in the present
paper only gives a strong rate of $p/2(p+1)$ and hence it is not
optimal. The reason for this is explained in \cite{KLS} with the fact
that the Green's function of the leap-frog scheme coincides at the
meshpoints with the Green's function for the wave equation, which is
not true for the present scheme. The latter paper, \cite{KLS}, studies
spatial discretization with finite elements in several dimensions with
findings in agreement (in one dimension) with the finite difference
spatial approximation studied in \cite{San}.

In connection to hyperbolic equations \cite{MR2268663} should be
mentioned, where the authors are concerned with weak as well as strong
convergence of a time discretization scheme for a nonlinear
stochastic Schr\"odinger equation.

Finally, in Section \ref{sec:para}, we apply Theorem \ref{main} to the
backward Euler in time and finite element in space approximation of
the linearized Cahn-Hilliard-Cook equation. The reason for doing this
is twofold. First, we demonstrate that the general error
representation formula is useful in the parabolic setting as
well. Second, the stochastic heat equation, which is the canonical
parabolic equation, has been studied in
\cite{debusscheprintems}. While our general approach would certainly
be applicable to the heat equation as well, it would just reprove a
known result, maybe with a more transparent proof, see Remark
\ref{rem:dp}. Similarly to the wave equation, also here we find that
the rate of weak convergence is twice that of the strong convergence
\cite{LM} under essentially the same assumptions.

The literature on weak convergence for parabolic equations is
richer. We have already mentioned \cite{debusscheprintems} that proves
results for fully discrete schemes of the linear stochastic heat
equation. In \cite{GKL} spatial, finite element schemes are considered
for the same equation. Such schemes are also studied for the linear
Cahn-Hilliard-Cook equation as well as the linear heat equation in
\cite{KLLweak}. Semidiscrete temporal schemes are investigated in
\cite{MR2189620} for the linear heat equation and in \cite{debussche}
for the nonlinear heat equation in one dimension.  The techniques of
the latter paper are extended to spatially semidiscrete schemes in
multiple dimensions in \cite{ALweak}. A recent paper,
\cite{LindnerSchilling}, successfully uses the methods of
\cite{debusscheprintems} to study a linear parabolic SPDE driven by
impulsive noise instead of a Wiener noise. This indicates the
possibility of extending the results of the present paper to larger
classes of noise.

%%%%%%%%%%%%%%%%%%%%%%%%%%%%%%%%%%%%%%%%%%%%%%%%%%%%%%%%%%%%%%%%%%%%%%%%%%%%%%%%%%%%%%%%%%%%%%%%%%%%%%%%%%%%%%%%%%%%%%%%%%%%%%%%%%%%%%%%%%%%%%%%%%%%%%%%%%%%%%%%%%%%
\section{Preliminaries}\label{pre}

Here we collect some background material from infinite-dimensional
stochastic analysis and stochastic PDEs. We use the semigroup approach
of DaPrato and Zabczyk and we refer to the monograph \cite{MR1207136}
for details and proofs.

Let $\cU$ and $\cH$ be real separable Hilbert spaces; we often denote
both their norms and scalar products by $\|\cdot\|$ and
$\langle\cdot,\cdot\rangle$ when the meaning is clear from the
context. We denote the space of bounded linear operators from $\cU$ to
$\cH$ by $\cB(\cU,\cH)$ and the $p$:th Schatten class of operators
from $\cU$ to $\cH$ by $\cL_p(\cU,\cH)$. They are Banach spaces for
all integers $p\geq 1$ and we will denote their norms by
$\|\cdot\|_{\cL_p(\cU,\cH)}$.  The operators in $\cL_1(\cU,\cH)$ are
also refered to as trace class operators and operators in
$\cL_2(\cU,\cH)$ as Hilbert-Schmidt operators. The space
$\cL_2(\cU,\cH)$ is a Hilbert space with inner product denoted
$\langle\cdot,\cdot \rangle_{\cL_2(\cU,\cH)}$. When the underlying
Hilbert spaces are understood from the context we will write
$\|\cdot\|_{\Tr}=\|\cdot\|_{\cL_1(\cU,\cH)}$,
$\|\cdot\|_{\HS}=\|\cdot\|_{\cL_2(\cU,\cH)}$ and
$\langle\cdot,\cdot\rangle_{\HS}=\langle\cdot,\cdot
\rangle_{\cL_2(\cU,\cH)}$ in order to -- we hope -- increase the
readability of the paper.

In case $\cH=\cU$ we write $\cB(\cU)=\cB(\cU,\cU)$ and
$\cL_p(\cU)=\cL_p(\cU,\cU)$ for short. If $T\in \cL_1(\cU)$ and
$\{e_k\}_{k=1}^\infty$ is an orthonormal basis of $\cU$, then the
trace of $T$,
\begin{equation*}
\Tr(T):=\sum_{k=1}^\infty \langle T e_k,e_k\rangle_{\cU},
\end{equation*}
is a well defined number, independent of the choice of orthonormal
basis. Below we state a number of properties of Schatten class
operators. For proofs and definitions we refer to, for example,
\cite[Appendix C]{MR1207136}, \cite{Lax} and \cite{Weid}.

If $T\in \cL_p(\cU,\cH)$, then its adjoint $T^*\in \cL_p(\cH,\cU) $ and
\begin{equation}\label{eq:trace0}
\|T\|_{\cL_p(\cU,\cH)}=\|T^*\|_{\cL_p(\cH,\cU)}.
\end{equation}
If $\cU=\cH$ and $p=1$, then also
\begin{equation}
  \label{eq:trace3}
  \Tr(T)=\Tr(T^*)
\end{equation}
and
\begin{equation}\label{eq:trace1}
|\Tr(T)|\leq\|T\|_{\Tr}.
\end{equation}
Further, if $T$ is selfadjoint and positive semidefinite, then
$\Tr(T)\ge0$ and \eqref{eq:trace1} holds with equality.

If $\cU_1$, $\cU_2$, and $\cH$ are separable Hilbert spaces and
$T\in\cL_p(\cU_2,\cH)$ and if $S_1\in\cB(\cU_1,\cU_2)$ and
$S_2\in\cB(\cH,\cU_1)$, then
\begin{equation}\label{eq:TrHS2}
\begin{aligned}
&\|TS_1\|_{\cL_p(\cU_1,\cH)}\le \|T\|_{\cL_p(\cU_2,\cH)}\|S_1\|_{\mathcal{B}(\cU_1,\cU_2)},\\
&\|S_2T\|_{\cL_p(\cU_2,\cU_1)}\le \|T\|_{\cL_p(\cU_2,\cH)}\|S_2\|_{\mathcal{B}(\cH,\cU_1)} .
\end{aligned}
\end{equation}
If $S\in\cB(\cH,\cU)$ and $T\in\cL_1(\cU,\cH)$, then we also have
\begin{align}
\Tr(TS)=\Tr(ST). \label{eq:trace2}
\end{align}
Moreover, if $T\colon \cU\rightarrow \cH$ and $T^*T\in\cL_1(\cU)$,
then $T\in\cL_2(\cU,\cH)$, $TT^*\in\cL_1(\cH)$ and
\begin{equation}\label{eq:TrHS0}
\begin{aligned}
\|T^*T\|_{\Tr}&=\Tr(T^*T)=\|T\|^2_{\HS}=\|T^*\|^2_{\HS}\\
&=\Tr(TT^*)=\|TT^*\|_{\Tr}.
\end{aligned}
\end{equation}
Finally, we note that if $T\in\cL_2(\cU,\cH)$ and $S\in
\cL_2(\cH,\cU)$, then $TS\in\cL_1(\cH)$ and
\begin{equation}\label{eq:TrHS1}
\|TS\|_{\Tr}\leq\|T\|_{\HS}\|S\|_{\HS}=(\Tr(TT^*)\Tr(SS^*))^{1/2}.
\end{equation}

To be able to compare various assumptions on the regularity of the
noise, where the regularity usually is measured in the trace or
Hilbert-Schmidt norms, we cite Theorem 2.1 in \cite{KLLweak}.

\begin{theorem}\label{thm:aq}
Assume that $Q\in\cB(\cH)$ is selfadjoint, positive semidefinite and
that $A$ is a densely defined, unbounded, selfadjoint, positive
definite, linear operator in $\cH$ with an orthonormal basis of
eigenvectors. Then, for $s\in\mathbb{R}$, $\alpha>0$, we have
\begin{align}
\nonumber
&\|A^{\frac{s}{2}}Q^{\frac{1}{2}}\|^2_{\HS}\le \|A^sQ\|_{\Tr}
\le \|A^{s+\alpha}Q\|_{\mathcal{B}(\cH)}\|A^{-\alpha}\|_{\Tr},\\
\label{c2}
&\|A^{\frac{s}{2}}Q^{\frac{1}{2}}\|^2_{\HS}\le \|A^{s+\frac{1}{2}}QA^{-\frac{1}{2}}\|_{\Tr}.
\end{align}
Furthermore, if $A$ and $Q$ have a common basis of eigenvectors, in particular, if $Q=I$, then
\begin{equation*}
\|A^{\frac{s}{2}}Q^{\frac{1}{2}}\|^2_{\HS}= \|A^sQ\|_{\Tr}=\|A^{s+\frac{1}{2}}QA^{-\frac{1}{2}}\|_{\Tr}.
\end{equation*}
\end{theorem}

Let $(\Omega,\cF,\mathbf{P})$ be a probability space and
$L_p(\Omega,\cH)$ denote the space of random variables $X\colon
(\Omega,\cF)\rightarrow (\cH,\Bor(\cH))$, where $\Bor(\cH)$ denotes
the Borel $\sigma$-algebra of the separable Hilbert space $\cH$, such
that
\begin{equation*}
\|X\|_{L_p(\Omega,\cH)}^p=\bE\big(\|X\|_{\cH}^p\big)
=\int_{\Omega}\|X(\omega)\|_{\cH}^p\,\dd \mathbf{P}(\omega)<\infty.
\end{equation*}
By ''strong convergence'' we mean norm convergence in $L_2(\Omega,\cH)$.

If $\{F(t)\}_{t\in[0,T]}$ is a family of (deterministic) bounded
linear operators from a Hilbert space $\cU$ to another $\cH$, then the
It\^o integral (also called the Wiener integral as the integrand is
deterministic)
\begin{equation*}
\int_0^TF(t)\,\dd W(t)
\end{equation*}
with respect to a $\cU$-valued $Q$-Wiener process is well defined if
\begin{equation}\label{ito}
\int_0^T\|F(t)Q^{1/2}\|_{\HS}^2\,\dd t
=\int_0^T\Tr(F(t)QF^*(t))\,\dd t<\infty.
\end{equation}
If \eqref{ito} holds then we have It\^o's isometry
\begin{equation}\label{eq:Itoiso}
\Big\|\int_0^TF(t)\,\dd W(t)\Big\|_{L_2(\Omega,\cH)}^2
=\int_0^T\|F(t)Q^{1/2}\|_{\HS}^2\,\dd t.
\end{equation}

The functionals $G$ in \eqref{eq:weakErr} are called test functions
and throughout this paper we will assume that they are mappings from
$\cH$ to $\R$ with bounded and continuous first and second Fr\'echet
derivatives. That is, they belong to the space
\begin{equation*}
\Cbb=\Cbb(\cH,\R)=\left\{G \in C^2(\cH,\R) :\|G\|_{\Cbb(\cH,\R)}<\infty \right\},
\end{equation*}
where
\begin{equation*}
\|G\|_{\Cbb(\cH,\R)}:=\sup_{x\in \cH}\|G'(x)\|_{\cH}+\sup_{x\in H}\|G''(x)\|_{\cB(\cH)}.
\end{equation*}
The derivatives $G'(x)$ and $G''(x)$ are identified with elements in
$\cH$ and $\cB(\cH)$, respectively, by the Riesz representation
theorem. Note that $\|\cdot\|_{\Cbb(\cH,\R)}$ is only a seminorm and
that we do not assume that $G$ itself is bounded.

%%%%%%%%%%%%%%%%%%%%%%%%%%%%%%%%%%%%%%%%%%%%%%%%%%%%%%%%%%%%%%%%%%%%%%%%%%%%%%%%%%%%%%%%%%%%%%%%%%%%%%%%%%%%%%%%%%%%%%%%%%%%%%%%%%%%%%%%%%%%%%%%%%%%%%%%%%%%%%%%%%%%%%%
\section{An error representation formula}\label{errep}
Following \cite{debusscheprintems} and \cite{KLLweak}, we start by
developing a representation of the weak error $e(T)$ in
\eqref{eq:weakErr} in the abstract setting.  To do this we use the
process $\{Y(t)\}_{t\geq 0}$ in \eqref{eq:Y}, being the unique weak
solution of the drift-free differential equation \eqref{eq:eqY} with
the important property that $Y(T)=X(T)$.  We also introduce the
auxiliary problem
\begin{align*}
  \dd Z(t)=E(T-t)B\,\dd W(t),~t\in(\tau,T];\quad Z(\tau)=\xi,
\end{align*}
where $\xi\in L_1(\Omega,\cH)$ is an $\mathcal{F}_\tau$-measurable random variable.
Its unique weak solution is given by
\begin{align}
  \label{eq:Z}
  Z(t,\tau,\xi)=\xi+\int_\tau^tE(T-s)B\,\dd W(s), \quad t\in[\tau,T].
\end{align}
We note that $Z(t,0,E(T)X_0)=Y(t)$ and that $Z(t,t,\xi)=\xi$.  For
$G\in C^2_\text{b}(\cH,\mathbb{R})$, we define the continuous function
$u\colon\cH\times [0,T]\to \mathbb{R}$ by
\begin{equation}\label{g}
u(x,t)={\bf E}\big(G(Z(T,t,x))\big).
\end{equation}
It follows from \eqref{eq:Z} and \eqref{g} that the partial
derivatives of $u$ are given by
\begin{equation}\label{eq:uderiv}
\begin{aligned}
u_x(x,t)={\bf E}\big(G'(Z(T,t,x))\big) , \\
u_{xx}(x,t)={\bf E}\big(G''(Z(T,t,x))\big).
\end{aligned}
\end{equation}
Hence,
\begin{equation}\label{eq:uderBound}
\begin{aligned}
\sup_{(x,t)\in\HH\times [0,T]}\|u_{x}(x,t)\|
\leq \|G\|_{\Cbb(\HH,\R)},\\
\sup_{(x,t)\in\HH\times [0,T]}\|u_{xx}(x,t)\|_{\cB(\HH)}
\leq\|G\|_{\Cbb(\HH,\R)}.
\end{aligned}
\end{equation}

It is known that $u$ is a solution to Kolmogorov's equation
\begin{align} \label{Kolmogorov}
  \begin{aligned}
&u_t(x,t)+\tfrac{1}{2}\Tr\big(u_{xx}(x,t)E(T-t)BQB^*E(T-t)^*\big)=0,
&&(x,t)\in \cH\times[0,T),\\
&u(x,T)=G(x),
&& x\in \cH,
  \end{aligned}
\end{align}
see, for example, \cite[Lemma 6.1.1]{MR1985790}, where we
note that $G\in \Cbb(\cH,\R)$ is enough for existence and $G\in
U\Cbb(\cH,\R)$ is only needed for uniqueness as the proofs of
\cite[Theorems 3.2.3 and 3.2.7]{MR1985790} show and the global
boundedness of $G$ is not needed by \cite[Remark 3.2.1]{MR1985790}.
Finally, the partial derivatives of $u$ in \eqref{Kolmogorov} are
continuous on $[0,T)\times \cH$.

Our key result is the following representation formula for the weak
error.

\begin{theorem}\label{main}
  Assume that \eqref{qt} and \eqref{qtdiscrete} hold and let
  $\{X(t)\}_{t\in[0,T]}$ be the unique mild solution \eqref{eq:X} of
  \eqref{spde} and that $\T{X}(T)$ can be represented as
  $\T{X}(T)=\T{Y}(T)$, where $\T{Y}$ is given by \eqref{eq:tildeY2}.

If $G\in \Cbb(\cH,\mathbb{R})$, then the weak error $e(T)$ in
\eqref{eq:weakErr} has the representation
\begin{equation}\label{e1}
\begin{split}
e(T)& ={\bf E}\int_0^1
\Big\langle u_x\big(Y(0)+s(\tilde{Y}(0)-Y(0)),0\big),\tilde{Y}(0)-Y(0) \Big\rangle\,\dd s \\
&\quad
+\tfrac{1}{2}{\bf E}\int_0^T
\Tr\Big( u_{xx}(\tilde{Y}(t),t)\mathcal{O}(t)\Big)\,\dd t,
\end{split}
\end{equation}
where
\begin{equation}\label{eq:F1}
\mathcal{O}(t)=\big(\tilde{E}(T-t)\T{B}+E(T-t)B\big)Q\big(\tilde{E}(T-t)\T{B}-E(T-t)\T{B}\big)^*,
\end{equation}
or
\begin{equation}\label{eq:F2}
\mathcal{O}(t)=\big(\tilde{E}(T-t)\T{B}-E(T-t)B\big)Q\big(\tilde{E}(T-t)\T{B}+E(T-t)B\big)^*.
\end{equation}
\end{theorem}

\begin{proof}
  As in \cite[Theorem 9.8]{MR1207136}, since $\xi$ is
  $\mathcal{F}_t$-measurable, we have that
\begin{equation}\label{eq:ux}
u(\xi,t)={\bf E}\Big(G(Z(T,t,\xi))\Big| \mathcal{F}_t\Big).
\end{equation}
Thus, by the law of double expectation,
\begin{equation} \label{eq:u1}
  {\mathbf E} \Big( u(\xi,t) \Big)
   ={\mathbf E} \Big(
  {\mathbf E} \Big( G(Z(T,t,\xi))\Big|{\mathcal F}_t \Big) \Big)
   ={\mathbf E}\Big( G(Z(T,t,\xi)) \Big) .
\end{equation}
Therefore, taking also into account that $X(T)=Y(T)$, it follows that
\begin{equation*}
 {\mathbf E} \Big( G(X(T)) \Big)
 ={\mathbf E} \Big( G(Y(T)) \Big)= {\mathbf E}\Big( G(Z(T,0,Y(0)) \Big)
 = {\mathbf E} \Big( u(Y(0),0) \Big)
\end{equation*}
and, since  $\T{X}(T)=\T{Y}(T)$, we also have that
\begin{align*}
   {\mathbf E} \Big( G(\T{X}(T) )\Big)
= {\mathbf E} \Big( G(\tilde{Y}(T) )\Big)
={\mathbf E}\Big( G(Z(T,T,\tilde{Y}(T)) \Big)
= {\mathbf E} \Big( u(\tilde{Y}(T),T) \Big).
\end{align*}
Hence,
\begin{align*}
 \begin{split}
  e(T)&= {\mathbf E} \Big( G(\T{X}(T)) -  G(X(T)) \Big)
       = {\mathbf E} \Big( u(\tilde{Y}(T),T) -  u(Y(0),0) \Big) \\
      &= {\mathbf E} \Big( u(\tilde{Y}(0),0) -  u(Y(0),0) \Big)
      +  {\mathbf E} \Big( u(\tilde{Y}(T),T) -  u(\tilde{Y}(0),0) \Big).
\end{split}
\end{align*}
For the first term we note that due to the differentiability of $u$ we
can write
\begin{equation*}
\begin{split}
&{\mathbf E} \Big( u(\tilde{Y}(0),0) -  u(Y(0),0) \Big)\\
&\qquad ={\bf E}\int_0^1
\Big\langle u_x\big(Y(0)+s(\tilde{Y}(0)-Y(0))\big),\tilde{Y}(0)-Y(0)
\Big\rangle\,\dd s.
\end{split}
\end{equation*}
For the second term, we use It\^o's formula (see \cite[Theorem
2.1]{Brez}, where, in contrast to \cite[Theorem 4.17]{MR1207136},
uniform continuity of the appearing derivatives on bounded subsets of
$ \cH\times [0,T)$ is not assumed) for $u(\tilde{Y}(t),t)$ on
$[0,T-\epsilon]$ and passing to the limit $\epsilon\to 0+$ using the
continuity of $u$ on $\cH\times [0,T]$ and the continuity of the paths
of $\tilde{Y}(t)$ on $[0,T]$.  Thus, taking also Kolmogorov's equation
\eqref{Kolmogorov} into account, we get
\begin{equation}\label{tr}
\begin{split}
&{\mathbf E} \Big( u(\tilde{Y}(T),T) -  u(\tilde{Y}(0),0) \Big)\\
&\quad={\mathbf E}\int_0^T \Big\{ u_t(\tilde{Y}(t),t)
\\ &\qquad +\tfrac{1}{2}\Tr\Big(u_{xx}(\tilde{Y}(t),t)
[\tilde{E}(T-t)\T{B}]Q[\tilde{E}(T-t)\T{B}]^*\Big) \Big\} \,\dd t\\
&\quad=\tfrac{1}{2}{\mathbf E}
\int_0^T \Tr\Big(u_{xx}(\tilde{Y}(t),t)
\big\{ [\tilde{E}(T-t)\T{B}]Q[\tilde{E}(T-t)\T{B}]^*\\
&\quad \quad -[E(T-t)B]Q[E(T-t)B]^*\big\} \Big) \,\dd t.
 \end{split}
\end{equation}
The operator $u_{xx}(\xi,r)$ is bounded for every $\xi$ and $r$ and
both $\tilde{E}(s)\T{B}Q[\tilde{E}(s)\T{B}]^*$ and $E(s)BQ[E(s)B]^*$
are of trace class for almost every $s$ by assumptions \eqref{qt} and
\eqref{qtdiscrete}. Hence, the trace above is well defined for almost
every $t$ since by \eqref{eq:TrHS2} with $p=1$,
\begin{equation*}
\begin{split}
\|u_{xx}(\xi,r) [E(s)B]Q[E(s)B]^*\|_{\Tr}
& \leq
\|u_{xx}(\xi,r)\|_{\cB(\HH)} \, \| [E(s)B]Q[E(s)B]^*\|_{\Tr}\\
&=\|u_{xx}(\xi,r)\|_{\cB(\HH)}\, \Tr\big( [E(s)B]Q[E(s)B]^*\big),
\end{split}
\end{equation*}
where the last step is \eqref{eq:trace1} with equality, which holds
since $[E(s)B]Q[E(s)B]^*$ is selfadjoint and positive
semidefinite. The same computation can be done with the term involving
$[\tilde{E}(s)\T{B}]Q[\tilde{E}(s)\T{B}]^*$.  Further, the
operator $u_{xx}(\xi,r) [E(s)B] Q[\tilde{E}(s)\T{B}]^*$ is also of
trace class for almost every $s$, since, by \eqref{eq:trace0},
\eqref{eq:TrHS2}, and \eqref{eq:TrHS1},
\begin{equation*}
\begin{split}
&\|u_{xx}(\xi,r) [E(s)B] Q[\tilde{E}(s)\T{B}]^*\|_{\Tr}\\
&\quad\le \|u_{xx}(\xi,r)\|_{\cB(\HH)} \,
\|[E(s)B] Q[\tilde{E}(s)\T{B}]^*\|_{\Tr} \\
&\quad\leq\|u_{xx}(\xi,r)\|_{\cB(\HH)} \,
\|[E(s)B]Q^{1/2}\|_{\HS} \, \|Q^{1/2}[\tilde{E}(s)\T{B}]^*\|_{\HS}\\
&\quad=\|u_{xx}(\xi,r)\|_{\cB(\HH)}\,
\|[E(s)B]Q^{1/2}\|_{\HS} \,\| [\tilde{E}(s)\T{B}]Q^{1/2}\|_{\HS}\\
&\quad=\|u_{xx}(\xi,r)\|_{\cB(\HH)} \Big(
\Tr\big([E(s)B]Q[E(s)B]^*\big)
\Tr\big([\tilde{E}(s)\T{B}]Q[\tilde{E}(s)\T{B}]^*\big)\Big)^{1/2}.
\end{split}
\end{equation*}
Therefore we may rewrite the operator in the trace in \eqref{tr} by
adding and subtracting $u_{xx}(\xi,r) [E(s)B] Q[\tilde{E}(s)\T{B}]^*$ to
get
\begin{align*}
&u_{xx}(\xi,r)\big\{[\tilde{E}(s)\T{B}]
Q[\tilde{E}(s)\T{B}]^*-[E(s)B]Q[E(s)B]^*\big\}\\
&\quad=u_{xx}(\xi,r)[\tilde{E}(s)\T{B}-E(s)B]Q[\tilde{E}(s)\T{B}]^*\\
&\qquad+u_{xx}(\xi,r)[E(s)B]Q[\tilde{E}(s)\T{B}-E(s)B]^*\\
&\quad=:O_1+O_2.
\end{align*}
Further, using \eqref{eq:trace3}, \eqref{eq:trace2}, and that $Q$ and
$u_{xx}(\xi,r)$ are selfadjoint, we obtain
\begin{align}
\Tr(O_1+O_2)&=\Tr(O_1)+\Tr(O_2)=\Tr(O_1) +\Tr(O_2^*)\nonumber\\
&=\Tr(O_1)+\Tr([\tilde{E}(s)\T{B}-E(s)B]Q[E(s)B]^*u_{xx}(\xi,r))\nonumber\\
&=\Tr(O_1)+\Tr(u_{xx}(\xi,r)[\tilde{E}(s)\T{B}-E(s)B]Q[E(s)B]^*)\nonumber\\
&= \Tr\Big(u_{xx}(\xi,r)[\tilde{E}(s)\T{B}-E(s)B]Q[\tilde{E}(s)\T{B}+E(s)B]^*\Big)\label{trrr}\\
&=\Tr\Big([\tilde{E}(s)\T{B}+E(s)B]Q[\tilde{E}(s)\T{B}-E(s)B]^*u_{xx}(\xi,r)\Big)\nonumber\\
&=\Tr\Big(u_{xx}(\xi,r)[\tilde{E}(s)\T{B}+E(s)B]Q[\tilde{E}(s)\T{B}-E(s)B]^*\Big). \label{trr}
\end{align}
Finally, by inserting \eqref{trrr} or \eqref{trr} into \eqref{tr} we
finish the proof.
\end{proof}

%%%%%%%%%%%%%%%%%%%%%%%%%%%%%%%%%%%%%%%%%%%%%%%%%%%%%%%%%%%%%%%%%%%%%%%%%%%%%%%%%%%%%%%%%%%%%%%%%%%%%%%%%%%%%%%%%%%%%%%%%%%%%%%%%%%%%%%%%%%%%%%%%%%%%%%%%%%%%%%%%%%%%%%%
\section{Application to the wave equation}\label{sec:hyper}
\subsection{A general error formula}
\label{wave}
In this section we apply the general result from Section \ref{errep}
to the stochastic wave equation \eqref{eq:wave}. As mentioned in the
Introduction the wave equation may be re-written in the form
\eqref{spde}. We start out by making this statement precise. At the
same time we introduce a framework for measuring the regularity of the
solutions and to perform a careful error analysis.  To this aim we let
$\cD$ denote a convex open bounded domain in $\R^d$ with polygonal
boundary $\partial \cD$, and equip $L_2(\cD)$ with the usual norm
$\Lnorm{\cdot}$ and inner product $\Linner{\cdot}{\cdot}$ and let
$\Lambda=-\Delta$ be the Laplace operator with
$\D(\Lambda)=H^2(\cD)\cap H^1_0(\cD)$.  To measure regularity we
introduce a scale of Hilbert spaces. Let
$\{(\lambda_j,\phi_j)\}_{j=1}^\infty$ be eigenpairs of $\Lambda$ with
a nondecreasing sequence of positive eigenvalues $\lambda_j$ and
corresponding orthonormal eigenvectors $\phi_j$. For $\alpha\in
\mathbb{R}$ we endow $\D(\Lambda^{\alpha/2})$ with inner product
\begin{align*}
  \quad\dHinner{ v}{w}{\alpha}
    := \Linner{ \Lambda^{\alpha/2} v}{ \Lambda^{\alpha/2} w}
    = \sum_{j=1}^\infty \lambda_j^\alpha \Linner{ v}{\phi_j}\Linner{ w}{\phi_j},
\end{align*}
and the corresponding norm
$\dHnorm{v}{\alpha}^2=\dHinner{v}{v}{\alpha}$. For $\alpha \geq 0$ the
space $\dH^{\alpha}$ now may be defined through
\begin{equation}\label{eq:hdot}
\dH^{\alpha}:=\{v\in L_2(\cD):\dHnorm{v}{\alpha}<\infty  \}.
\end{equation}
If $\alpha<0$, then $\dH^{\alpha}$ is taken to be the closure of
$L_2(\cD)$ with respect to $\dHnorm{\cdot}{\alpha}$.  It is notable
that with $\alpha>0$ the space $\dH^{-\alpha}$ may be identified with
the dual of $\dH^{\alpha}$ and that $\dH^{\alpha}\subset\dH^{\beta}$
if $\alpha\geq \beta$.  Further, it is known that $\dH^0=L_2(\cD),\
\dH^1=H_0^1(\cD),\ \dH^2=H^2(\cD)\cap H^1_0(\cD)$, see
\cite[Chapt.~3]{Thomeebook}. In addition we define the product space
\begin{equation}
\label{defHalfa}
 \HH^\alpha := \dH^\alpha \times \dH^{\alpha-1},\quad \alpha \in \IR,
\end{equation}
with inner product
\begin{align*}
\langle v,w \rangle_{\HH^\alpha}
=\dHinner{v_1}{w_1}{\alpha}+\dHinner{v_2}{w_2}{\alpha-1},
\end{align*}
where  $v=[v_1,v_2]^T$ and  $w=[w_1,w_2]^T$. The corresponding norms are
\begin{equation*}
  \Hnorm{v}{\alpha}^2\,
    := \dHnorm{ v_1 }{\alpha}^2
      + \dHnorm{ v_2 }{\alpha-1}^2.
\end{equation*}
We take $\cH$ to be the special case of \eqref{defHalfa} when
$\alpha=0$ with norm $\HHnorm{\cdot}=\Hnorm{\cdot}{0}$ and inner
product
$\HHinner{\cdot}{\cdot}=\Linner{\cdot}{\cdot}+\Linner{\Lambda^{-1/2}\cdot}{\Lambda^{-1/2}\cdot}$.
We now regard $\Lambda$ as an operator from $\dH^1$ to $\dH^{-1}$ as
$(\Lambda x)(y)=\langle \nabla x, \nabla y\rangle_{L_2(\D)}$ and let
$A$ and $B$ be defined by
\begin{equation}\label{eq:waveDef}
  A:=\begin{bmatrix}0&-I\\\Lambda&0\end{bmatrix}, \quad
  B:= \begin{bmatrix} 0\\ I \end{bmatrix}.
\end{equation}
Here $B$ is considered as an operator from $\dH^{-1}$ to $\HH$
and the domain of $A$ is
\begin{align*}
 \D(A) = \Big\{ x \in \HH:A x
  =\begin{bmatrix} -x_2\\ \Lambda x_1 \end{bmatrix}
    \in \HH=\dH^0 \times \dH^{-1}\Big\}
       = \HH^1=\dH^1 \times \dH^0.
\end{align*}
We have some freedom in defining $\cU$ but a natural choice is
$\cU=\dH^0=L_2(\cD)$. Thus, we consider the process $\{W(t)\}_{t\geq
  0}$ to be a $Q$-Wiener process in $L_2(\cD)$ and thus $Q$ to be
bounded, selfadjoint and positive semidefinite on $\cU=L_2(\cD)$. Note
that $\dH^0\subset \dH^{-1}$, and therefore $B\in
\mathcal{B}(\cU,\cH)$. It is well known that the operator $-A$ is the
generator of a strongly continuous semigroup $E(t) = \ee^{-tA}$ on
$\HH$, in fact, a unitary group, that can be written as
\begin{align}
\label{eq:exgroup}
E(t)=\ee^{-t A}=
  \begin{bmatrix}
   C(t) & \Lambda^{-1/2} S(t)\\
   -\Lambda^{1/2}S(t) & C(t)
   \end{bmatrix},
\end{align}
where $C(t)=\cos(t\Lambda^{1/2})$ and $S(t)=\sin(t\Lambda^{1/2})$ are
the so-called cosine and sine operators.

With these definitions \eqref{spde} becomes the stochastic wave
equation \eqref{eq:wave} with solution $\bm{X_1(t),X_2(t)}^T\in \HH$
if also the initial value $X_0=\bm{X_{0,1},X_{0,2}}^T\in \HH$ and
\eqref{qt} holds.

We are know ready to use the framework set forth in Section
\ref{errep}, starting by assuming as little as possible about the type
of perturbations of $X(t)$. To get a result that enables error
analysis of the full solution $X$, as well as either of the coordinates
$X_1$ or $X_2$, we make three additional, rather weak assumption on the
data of the problem. First, we take $G$ to be the composition of a
function $g \in \Cbb(\mathcal{V},\R)$, where $\mathcal{V}$ is a real
separable Hilbert space, with an operator $L\in \cB(\HH,\VV)$, i.e.,
\begin{equation*}
G(v)=g(Lv),\quad v=[v_1,v_2]^T\in \HH.
\end{equation*}
It is clear that $G\in \Cbb(\HH,\R) $. Typically we have $\VV=\HH$ and
$L=I$, or $\VV=\dH^0$ and $Lv=P^1v=v_1$ (projection to the first
coordinate), or possibly $\VV=\dH^{-1}$ and $Lv=P^2v=v_2$ (projection
to the second coordinate).  Second, since in all interesting cases
the function $\T{X}_0$ will be related to $X_0$ we will assume that
this relation can be described by an operator $\T{P}\in\cB(\HH)$
through
\begin{equation*}
\T{X}_0=\T{P}X_0.
\end{equation*}
In the cases we consider below, $\T{P}$ will be the orthogonal
projection onto a finite element subspace of $\cH$ or the identity
operator but it could also be an interpolation operator. The third
assumption we make is that $\T{B}=\T{P}B$. This is unnecessarily
restrictive in general but it suffices for the purposes of this paper
and it also makes the presentation clearer.  We begin with a technical
result.

\begin{lemma}\label{lem:assumpEquality} If $E(t)$ is given by
  \eqref{eq:exgroup}, then the following four statements are
  equivalent.
\begin{enumerate}
\renewcommand{\theenumi}{\roman{enumi}}
\item \label{itm:P1} $\Tr(\Lambda^{-1/2}Q\Lambda^{-1/2})<\infty$.
\item \label{itm:P1b} $\|\Lambda^{-1/2}Q^{1/2} \|_{\HS}<\infty$.
\item \label{itm:P2} $\|\Lambda^{-1/2}Q\Lambda^{-1/2}\|_{\Tr}<\infty $.
\item \label{itm:P3} $\Tr(\int_0^T E(t)BQB^*E(t)^*\,\dd t)< \infty$
  for some, hence all, $T>0$.
\end{enumerate}
If either of them holds, then
\begin{equation}\label{eq:assumpEquality}
\Tr\Big(\int_0^T E(t)BQB^*E(t)^*\,\dd t\Big)=T\Tr(\Lambda^{-1/2}Q\Lambda^{-1/2}).
\end{equation}
\end{lemma}

\begin{proof}
The operator $\Lambda^{-1/2}Q\Lambda^{-1/2}$ is selfadjoint and
positive semidefinite on $\dot{H}^0$. Hence, if the trace is finite
it equals the trace norm and the other way around. This implies that
\eqref{itm:P1} $\Leftrightarrow$ \eqref{itm:P2}. We know assume that
\eqref{itm:P3} holds. By monotone convergence,
\begin{equation}\label{eq:trFub}
\Tr\Big(\int_0^T E(t)BQB^*E(t)^*\,\dd t\Big)=\int_0^T\Tr\Big( E(t)BQB^*E(t)^*\Big)\,\dd t
\end{equation}
 and by \eqref{eq:TrHS0} and since $E(t)^*=E(t)^{-1}$ we have
\begin{align*}
&\Tr_{\HH}( E(t)BQB^*E(t)^*)
=\Tr_{\dH^0}(Q^{1/2}B^*E(t)^* E(t)BQ^{1/2})
=\Tr_{\dH^{0}}(Q^{1/2}B^*BQ^{1/2})
\\ & \qquad
=\|BQ^{1/2}\|^2_{\cL_2(\dH^0,\HH)}=
\|[0,Q^{1/2}]^T\|^2_{\cL_2(\dH^0,\HH)}
=\|Q^{1/2}\|^2_{\cL_2(\dH^0,\dH^{-1})}
\\ &\qquad
=\|\Lambda^{-1/2}Q^{1/2}\|^2_{\cL_2(\dH^0)}
=\Tr_{\dH^0}(Q^{1/2}\Lambda^{-1}Q^{1/2})=\Tr_{\dH^0}(\Lambda^{-1/2}Q\Lambda^{-1/2}).
\end{align*}
Therefore, it follows that the integrand on the right-hand side of \eqref{eq:trFub} is
constant and is equal to $\Tr_{\dH^0}(\Lambda^{-1/2}Q\Lambda^{-1/2})$,
which implies \eqref{eq:assumpEquality}. Therefore \eqref{itm:P1} must be
true. This argument is reversible so \eqref{itm:P1}
$\Leftrightarrow$ \eqref{itm:P3}. But from this computation it is also
evident that
$$\|\Lambda^{-1/2}Q^{1/2}\|^2_\HS=\Tr(\Lambda^{-1/2}Q\Lambda^{-1/2});$$
that is, that \eqref{itm:P1} $\Leftrightarrow$ \eqref{itm:P1b}.
\end{proof}

\begin{corollary}\label{cor:Qassump}
  If $\|\Lambda^{\beta-1/2}Q\Lambda^{-1/2}\|_{\Tr}<\infty$ for some
  $\beta\geq 0$, then \eqref{itm:P3} holds and \eqref{spde} has a unique weak solution $\{X(t)\}_{t\in [0,T]}$.
\end{corollary}

\begin{proof}
By \eqref{eq:trace1} and \eqref{eq:TrHS2} we have that
\begin{align*}
|\Tr(\Lambda^{-1/2}Q\Lambda^{-1/2})|
&
= \|\Lambda^{-1/2}Q\Lambda^{-1/2}\|_{\Tr}
= \|\Lambda^{-\beta}\Lambda^{\beta-1/2}Q\Lambda^{-1/2}\|_{\Tr}\\
& \le \|\Lambda^{-\beta}\|_{\mathcal{B}(\dH^0)} \|\Lambda^{\beta-1/2}Q\Lambda^{-1/2}\|_{\Tr}.
\end{align*}
The result now follows by Lemma \ref{lem:assumpEquality}.
\end{proof}

Since $\|\Lambda^{\frac{\beta-1}2}Q^{\frac12}\|_{\HS}^2\le
\|\Lambda^{\beta-1/2}Q\Lambda^{-1/2}\|_{\Tr} $ by \eqref{c2}, we may
conclude that under the assumption in the previous corollary, we actually
have mean-square regularity of order $\beta$.  This follows
from the following theorem, which is quoted from \cite[Theorem
3.1]{KLS}.

\begin{theorem}\label{cor:QassumpA}
  If $\|\Lambda^{\frac{\beta-1}2}Q^{\frac12}\|_{\HS}^2<\infty$,
  $X_0\in L_2(\Omega,\cH^\beta)$ for some $\beta\geq 0$ and $A,B$ as
  in \eqref{eq:waveDef}, then the mild solution \eqref{eq:X} satisfies
  \begin{align*}
    \|X(t)\|_{L_2(\Omega,\cH^\beta)}
   \le C\big( \|X_0\|_{L_2(\Omega,\cH^\beta)} +
   T^{1/2} \|\Lambda^{\frac{\beta-1}2}Q^{\frac12}\|_{\HS}\big).
  \end{align*}
\end{theorem}

We will need the following result on the H\"older continuity of
$E(t)$. It will put an ultimate limit on the convergence rate that one
can achieve with respect to time-stepping.

\begin{lemma}\label{lem:holderCont}
If $\{E(t)\}_{t\geq 0}$ is the semigroup in \eqref{eq:exgroup}, then
\begin{equation*}
\|(E(t)-E(s))x\|\leq C|t-s|^\alpha\|x\|_{\HH^\alpha},\quad
x\in\HH^\alpha,\ t,s\geq 0,~\alpha\in[0,1].
\end{equation*}
\end{lemma}

\begin{proof}
The operator $E(t)$ is bounded on $\HH$ so the statement is true for
$\alpha=0$. Let $\alpha=1$ and $ x=[x_1,x_2]^T\in \HH^1$. Then
\begin{equation*}
\begin{split}
\|(E(t)-E(s))x\|^2
&=\|(C(t)-C(s))x_1+(S(t)-S(s))\Lambda^{-{1/2}}x_2\|^2_{\dH^0}\\
& \quad
+\|(S(s)-S(t))\Lambda^{1/2}x_1+(C(t)-C(s))x_2\|^2_{\dH^{-1}}=:A_1+A_2.
\end{split}
\end{equation*}
By the triangle inequality and the definition of the norm on
$\dH^{-1}$, we have for the last term that
\begin{equation*}
\begin{aligned}
A_2
&\leq 2\|\Lambda^{-1/2}(S(t)-S(s))\Lambda^{1/2}x_1\|^2_{\dH^0}
+2\|\Lambda^{-1/2}(C(t)-C(s))x_2\|^2_{\dH^{0}}\\
&= 2\|(S(t)-S(s))x_1\|^2_{\dH^0}+2\|(C(t)-C(s))\Lambda^{-{1/2}}x_2\|^2_{\dH^{0}}.
\end{aligned}
\end{equation*}
Since $x_2\in \dH^{0}$ it follows that $\Lambda^{-1/2}x_2\in
\dH^{1}$. Hence we only need to investigate the H\"older continuity of
$C$ and $S$ as functions from $[0,T]$ to $\cB(\dH^{1},\dH^{0})$. To
this aim we note that for real $y$ the inequality
\begin{equation}\label{eq:trigBnd}
|\sin(ty)-\sin(sy)|\leq |t-s||y|
\end{equation}
holds. It follows that for $\xi \in \dH^1$ we have
$\|(S(t)-S(s))\xi\|_{\dH^0}\leq|t-s|\|\xi\|_{\dH^1}$. Indeed,
\begin{equation*}
\begin{split}
\|(S(t)-S(s))\xi\|_{\dH^0}^2
&= \sum_{j=1}^\infty\la(S(t)-S(s))\xi,\phi_j\ra_{\dH^0}^2\\
&=\sum_{j=1}^\infty(\sin(t\lambda_j^{1/2})-\sin(s\lambda_j^{1/2}))^2\la\xi,\phi_j\ra_{\dH^0}^2\\
&\leq \sum_{j=1}^\infty(t-s)^2\lambda_j\la\xi,\phi_j\ra_{\dH^0}^2=(t-s)^2\dHnorm{\xi}{1}^2.
\end{split}
\end{equation*}
The inequality \eqref{eq:trigBnd} holds also with $\sin$ replaced by
$\cos$. Thus the statement of the lemma holds also for
$\alpha=1$. The intermediate case follows by interpolation.
\end{proof}

We are now ready to prove a weak error bound for perturbations
of the stochastic wave equation.

\begin{theorem}\label{thm:waveGen}
  Assume that $\{X(t)\}_{t\in[0,T]}$ is the mild solution \eqref{eq:X}
  of the stochastic wave equation \eqref{spde} with
  $\norm{\Tr}{\Lambda^{\beta-1/2}Q\Lambda^{-1/2}}<\infty$ and $X_0\in
  L_1(\Omega,\HH^{2\beta})$ for some $\beta\geq0$. Assume further that
  $\T{X}(T)$ can be represented as $\T{X}(T)=\T{Y}(T)$, where
  $\T{Y}(t)$ is given by \eqref{eq:tildeY2} with $\T{X}_0=\T{P}X_0,\
  \T{P}\in \cB(\HH)$ and $\T{B}=\T{P}B$ and such that
  \eqref{qtdiscrete} holds. Let $g \in \Cbb(\VV,\R) $ and $L\in
  \cB(\cH,\VV)$.  Define
\begin{align}
\label{eq:finTrA}
K_1&:=\sup_{t\in[0,T]}\norm{\cH}{\T{E}(t)\T{P}} , \\
\label{eq:finTr}
K_2&:=\norm{\Tr}{\Lambda^{\beta-1/2}Q\Lambda^{-1/2}}.
\end{align}
Then there is
$C=C(T,\|L\|_{\cB(\cH,\VV)},\|X_0\|_{L_1(\Omega,\HH^{2\beta})},\|g\|_{\Cbb},K_1,K_2)$
such that
\begin{equation}\label{eq:weakErrGen}
\Big|\bE\left(g(L\T{X}(T))-g(LX(T))\right)\Big|
\leq C \sup_{t\in[0,T]}\norm{\cB(\HH^{2\beta},\VV)}{L(\tilde{E}(t)\T{P}-E(t))}.
\end{equation}
\end{theorem}

We want to emphasize that this theorem reduces the problem of proving
weak error estimates for the stochastic wave equation to proving
deterministic error estimates for the approximation of the semigroup
or, to be precise, to find a bound for
\begin{equation*}
\sup_{t\in[0,T]}\norm{\cB(\HH^{2\beta},\VV)}{L(\tilde{E}(t)\T{P}-E(t))}.
\end{equation*}

\begin{proof}
  First note that \eqref{eq:finTr} implies \eqref{qt} by Corollary
  \ref{cor:Qassump} and hence we may use Theorem \ref{main} with
  $G(X)=g(LX)$. To this aim we note that $G'(X)=L^{*}g'(LX)$ and
  $G''(X)=L^{*}g''(LX)L$. The terms in \eqref{e1} will be estimated in
  order of appearance with $\mathcal{O}$ as in \eqref{eq:F2}. For the
  first term, by \eqref{eq:uderiv}, \eqref{eq:ux} and the fact that
  both $Y(0)$ and $\tilde{Y}(0)$ are $\cF_0$-measurable, we have
\begin{align*}
&\Big|\mathbf{E} \Big(\int_0^1
   \Big\langle u_x\big(Y(0)+s(\tilde{Y}(0)-Y(0)),0\big),
    \tilde{Y}(0)-Y(0) \Big\rangle \,\dd s  \Big)\Big|\\
 &\quad=\Big|\bE\Big(\int_0^1
\Big\langle
\bE \Big(
g'\big(LZ(T,0,Y(0)+s(\tilde{Y}(0)-Y(0)))\big)\Big|{\mathcal F}_0\Big),
L(\tilde{Y}(0)-Y(0))\Big\rangle\,\dd s\Big)\Big|\\
 &\quad \le \sup_{x\in \VV}\|g'(x)\| \,
  \mathbf{E}\big(\|L(\tilde{Y}(0)-Y(0))\|\big)\\
 &\quad = \sup_{x\in \VV}\|g'(x)\| \,
  \mathbf{E}\big(\|L(\T{E}(T)\T{X}_0-E(T)X_0)\|\big)\\
 &\quad \leq \|g\|_{\Cbb} \,
  \|L(\T{E}(T)\T{P}-E(T))\|_{\cB(\HH^{2\beta},\VV)}\|X_0\|_{L_1(\Omega,\HH^{2\beta})}.
\end{align*}

For the second term we note that with
\begin{align*}
\mathcal{O}^-(t)&=\tilde{E}(T-t)\T{B}-E(T-t)B=(\tilde{E}(T-t)\T{P}-E(T-t))B,\\
\mathcal{O}^+(t)&=\tilde{E}(T-t)\T{B}+E(T-t)B=(\tilde{E}(T-t)\T{P}+E(T-t))B,
\end{align*}
and by \eqref{eq:uderiv}, \eqref{eq:ux}, we may write
\begin{align*}
&\mathbf{E}\int_0^T\Tr\Big(u_{xx}\big(\tilde{Y}(t),t\big)
\mathcal{O}^+(t)Q\mathcal{O}^-(t)^*\Big)\,\dd t\\
&\qquad
=\mathbf{E}\int_0^T\Tr\Big(\mathbf{E}\big(L^*g''\big(LZ(T,t,\tilde{Y}(t))\big)
L\big|\mathcal{F}_t\big) \mathcal{O}^+(t)Q\mathcal{O}^-(t)^*\Big)\,\dd t.
\end{align*}
Using \eqref{eq:trace2}, \eqref{eq:trace1} and \eqref{eq:TrHS2} we
bound the integrand above as follows:
\begin{align*}
& \Big|\Tr\Big(\mathbf{E}
\big(L^*g''(LZ(T,t,\tilde{Y}(t)))L\big|\mathcal{F}_t\big)
\mathcal{O}^+(t)Q\mathcal{O}^-(t)^*\Big)\Big|\\
&\quad =\Big|\Tr\Big(\mathbf{E}
\big(g''(LZ(T,t,\tilde{Y}(t)))\big|\mathcal{F}_t\big)
L \mathcal{O}^+(t)Q\mathcal{O}^-(t)^*L^*\Big)\Big|\\
&\quad
\leq \sup_{x\in \VV}\|g''(x)\|_{\cB(\VV)}
\| L \mathcal{O}^+(t)Q\mathcal{O}^-(t)^*L^*\|_{\Tr}.
\end{align*}
Here we have $\sup_{x\in \VV}\|g''(x)\|_{\cB(\VV)}\le
\|g\|_{\Cbb} $ and by \eqref{eq:trace0} and \eqref{eq:TrHS2},
\begin{align*}
&\| L \mathcal{O}^+(t)Q\mathcal{O}^-(t)^*L^*\|_{\Tr}
=\| L \mathcal{O}^+(t)\Lambda^{1/2}\Lambda^{-1/2}Q\mathcal{O}^-(t)^*L^*\|_{\Tr}\\
& \quad \le \| L \mathcal{O}^+(t)\Lambda^{1/2}\|_{\cB(\dH^0,\VV)}\,
\| \Lambda^{-1/2}Q\mathcal{O}^-(t)^*L^*\|_{\Tr}\\
& \quad =\| L \mathcal{O}^+(t)\Lambda^{1/2}\|_{\cB(\dH^0,\VV)}\,
\| L\mathcal{O}^-(t)Q\Lambda^{-1/2}\|_{\Tr}\\
& \quad = \| L \mathcal{O}^+(t)\Lambda^{1/2}\|_{\cB(\dH^0,\VV)}\,
\| L\mathcal{O}^-(t)\Lambda^{1/2-\beta}\Lambda^{\beta-1/2}Q\Lambda^{-1/2}\|_{\Tr}\\
& \quad \le \| L \mathcal{O}^+(t)\Lambda^{1/2}\|_{\cB(\dH^0,\VV)}\,
\|L\mathcal{O}^-(t)\Lambda^{1/2-\beta} \,
\|_{\cB(\dH^0,\VV)}\|\Lambda^{\beta-1/2}Q\Lambda^{-1/2}\|_{\Tr}.
\end{align*}
The first factor can be estimated as
\begin{align*}
  &\| L \mathcal{O}^+(t)\Lambda^{1/2}\|_{\cB(\dH^0,\VV)}
  =\| L (\tilde{E}(T-t)\T{P}+E(T-t))
     B\Lambda^{1/2}\|_{\cB(\dH^0,\VV)} \\
&\qquad
\le\|L\|_{\cB(\cH,\VV)}
\big (\|\tilde{E}(T-t)\T{P}\|_{\cB(\cH,\cH)}
+\|E(T-t)\|_{\cB(\cH,\cH)} \big)
\| B \Lambda^{1/2}\|_{\cB(\dH^0,\cH)}\\
& \qquad \le \|L\|_{\cB(\cH,\VV)}(K_1+1),
\end{align*}
because $\|E(s)\|_{\cB(\cH,\cH)} =1=
\| B \Lambda^{1/2}\|_{\cB(\dH^0,\cH)}$.
Similarly, the middle factor may be bounded by
\begin{align*}
\big\|L \big(\tilde{E}(T-t)\T{P}-E(T-t)\big)\big\|_{\cB(\HH^{2\beta},\VV)} \,
\| B \Lambda^{1/2-\beta}\|_{\cB(\dH^0,\cH^{2\beta})},
\end{align*}
where $\| B \Lambda^{1/2-\beta}\|_{\cB(\dH^0,\cH^{2\beta})}=1$.
The third term is $K_2$. Thus, we conclude that
\begin{align*}
&\Big|\mathbf{E}\int_0^T\Tr\Big(u_{xx}\big(\tilde{Y}(t),t\big)
\mathcal{O}^-(t)Q \mathcal{O}^+(t)^*\Big)\,\dd t\Big|\\
&\qquad \leq C\int_0^T
\big\| L \big(\tilde{E}(T-t)\T{P}-E(T-t)\big)\big\|_{\cB(\HH^{2\beta},\VV)}\,\dd t\\
&\qquad \leq C\,T\sup_{t\in[0,T]}
\big\| L \big(\tilde{E}(t)\T{P}-E(t)\big)\big\|_{\cB(\HH^{2\beta},\VV)},
\end{align*}
and the proof is complete.
\end{proof}
\begin{remark}
We briefly comment on the abstract condition $\norm{\Tr}{\Lambda^{\beta-1/2}Q\Lambda^{-1/2}}<\infty$ of Theorem \ref{thm:waveGen}.
If $Q=I$, then we must have $\Tr(\Lambda^{\beta-1})<\infty$ and hence taking the asymptotics of the eigenvalues of $\Lambda$ into account
we conclude that $\frac{2}{d}(\beta-1)<-1$; that is, $\beta<1-\frac{d}{2}$. Hence $d=1$ and $\beta<\frac12$. In general, using \eqref{eq:TrHS2}, we have that
$$
\norm{\Tr}{\Lambda^{\beta-1/2}Q\Lambda^{-1/2}}\le \Tr(\Lambda^{\beta-1/2-s})\|\Lambda^sQ\Lambda^{-\frac12}\|_{\mathcal{B}(\cU)}.
$$
Thus, if $\frac{2}{d}(\beta-\frac12 -s)<-1$; that is, $\beta<s+\frac12 -\frac{d}{2}$, then
$\norm{\Tr}{\Lambda^{\beta-1/2}Q\Lambda^{-1/2}}<\infty$ provided that $\Lambda^sQ\Lambda^{-\frac12}$ is a bounded linear operator on $\cU$.
\end{remark}

%%%%%%%%%%%%%%%%%%%%%%%%%%%%%%%%%%%%%%%%%%%%%%%%%%%%%%%%%%%%%%%%%%%%%%%%%%%%%%%%%%%%
%
%
%%%%%%%%%%%%%%%%%%%%%%%%%%%%%%%%%%%%%%%%%%%%%%%%%%%%%%%%%%%%%%%%%%%%%%%%%%%%%%%%%%%%
\subsection{Weak convergence of temporally semidiscrete
  schemes} \label{sec:semitime}

We begin by applying Theorem \ref{thm:waveGen} to semidiscrete
approximation schemes where the discretization is with respect to
time. We will use results on so-called $I$-stable rational
approximations considered in \cite{BrennerThomee}. An $I$-stable
rational approximation of order $p$ is a rational function $R$ such
that \eqref{eq:ratApp} holds. A class of such functions are
constructed in \cite{BakBram} and analyzed further in connection to
oscillation equations in \cite{NorsettWanner}. It contains the
implicit Euler method ($p=1$) and, which is important since it
preserves energy for the wave equation, the Crank-Nicolson method
($p=2$).

For these functions the operators $E_k=R(kA)$, $k>0$, are well defined on $\HH$ and they are
contractions, and hence stable, as $-A$ generates a unitary group. Here $k=T/N,\ N\in
\mathbb{N}$, is the time step. We
can approximate the solution of \eqref{spde} on the uniform grid
$t_j=jk,\ j=0,\ldots,N$, by the solution of the
difference equation
\begin{equation}\label{eq:spdeRA}
X^j_k=E_k(X^{j-1}_k+B\Delta W^j),\ j=1,\ldots, N; \quad X^0_k=X_0,
\end{equation}
given by
\begin{equation*}
X^n_k=E_k^nX_0+\sum_{j=1}^{n}E_k^{n-j+1}B\Delta W^{j},\quad n\geq 1,
\end{equation*}
where $\Delta W^j=W(t_{j})-W(t_{j-1})$. We want to define a process
$\{\T{Y}_k(t)\}_{t\in[0,T]}$ of the form \eqref{eq:tildeY2} that is as
close as possible to \eqref{eq:Y} and such that $X_k^N=\T{Y}_k(T)$. To
this aim we first define a new discrete process
\begin{equation}\label{eq:driftfreeDisc}
Y_k^n=E_k^{N-n}X^n_k=E^N_kX_0+\sum_{j=1}^{n}E_k^{N-j+1}B\Delta W^{j}.
\end{equation}
Clearly $Y_k^N=X_k^N$. In order to make a piecewise constant time
interpolation of \eqref{eq:driftfreeDisc} we introduce the time
intervals $I_j=[t_{j-1},t_j)$ for $j=1,\ldots, N$ and
$I_{N+1}=\{t_N\}=\{T\}$. With $\chi$ being the indicator function we
then write
\begin{equation*}
\T{E}_k(T-t)=\sum_{j=1}^{N+1}E_k^{N-j+1}\chi_{I_j}(t).
\end{equation*}
It may easily be checked that this corresponds to writing
\begin{equation*}
\T{E}_k(t)=\sum_{j=0}^{N}E_k^{j}\chi_{\widehat{I}_j}(t)
\end{equation*}
with $\widehat{I}_{0}=\{t_0\}=\{0\}$ and $\widehat{I}_{j}=(t_{j-1},t_j]$ for
$j=1,\ldots, N$. We finally define
\begin{equation} \label{eq:Yktilde}
\T{Y}_k(t):=\T{E}_k(T)X_0+\int_0^t\T{E}_k(T-s)B\,\dd W(s).
\end{equation}
The process $\T{Y}_k$ has the desired properties and, in addition,
$\T{Y}_k(t_n)=Y_k^n$. To apply Theorem \ref{thm:waveGen} it remains to
show that \eqref{qtdiscrete} holds with $\T{E}(t)=\T{E}_k(t)$ and
$\T{P}=I$ (hence $\T{B}=B$). This is indeed the case as soon as we are
guaranteed a weak solution of \eqref{spde}, as stated in the following
lemma.

\begin{lemma}\label{lem:EkBound}
  If $\T{E}(t)=\T{E}_k(t)$ and $\T{B}=B$, then \eqref{qt} implies
  \eqref{qtdiscrete}.
\end{lemma}
\begin{proof}
  If \eqref{qt} holds then the trace of
  $\Lambda^{-1/2}Q\Lambda^{-1/2}$ is finite by Lemma
  \ref{lem:assumpEquality}. Thus, for all $t\ge 0$, using
  \eqref{eq:trace1} and \eqref{eq:TrHS2}, it follows that
 \begin{equation*}
\begin{split}
\Tr(\T{E}_k(t)BQB^*\T{E}_k^*(t))
&\leq \|\T{E}_k(t)\|_{\cB(\cH)}^2 \|BQB^*\|_{\Tr}\\
&\le \|BQB^*\|_{\Tr}= \Tr(BQB^*)=\Tr(\Lambda^{-1/2}Q\Lambda^{-1/2})<\infty,
\end{split}
\end{equation*}
where the last equality is shown in the proof of Lemma
\ref{lem:assumpEquality} as $$\Tr(BQB^*)=\Tr(Q^{1/2}B^*BQ^{1/2})$$ by
\eqref{eq:TrHS0}.  The statement of the lemma follows by the monotone
convergence theorem again as in the proof of Lemma
\ref{lem:assumpEquality}.
\end{proof}

In order to make use of Theorem \ref{thm:waveGen} we need a bound on
$E(t)-\T{E}_k(t)$. The results in \cite{BrennerThomee} are concerned
with the difference at the grid points. With our notation their
conclusion reads that, for $x\in \HH^{p+1}$,
\begin{equation}\label{eq:raterr}
\|(E(t_n)-E_k^n)x\|\leq Ct_nk^p\|x\|_{\HH^{p+1}}.
\end{equation}
As already mentioned, the conditions in \eqref{eq:ratApp} ensures that
the operator $E_k^n$ is a contraction on $\HH$ for any $n\geq 0$, so
\eqref{eq:raterr} can be extended to fractional order by
interpolation, i.e.,
\begin{equation}\label{eq:raterrfrac}
\|(E(t_n)-E_k^n)x\|\leq
Ct_nk^{\alpha\frac{p}{p+1}}\|x\|_{\HH^{\alpha}},
\quad \alpha\in[0,p+1].
\end{equation}
For our purposes, it is not enough to consider only the grid points,
but fortunately a global error estimate follows easily.

\begin{lemma}\label{lem:supE}
  For the operators $\T{E}_k(t)$ and $E(t)$ defined above, we have
  that
\begin{equation*}
\sup_{t\in[0,T]}\|\tilde{E}_k(t)-E(t)\|_{\mathcal{B}(\HH^\alpha,\HH)}
\leq C(T)k^{\min(\alpha\frac{p}{p+1},1)},~k>0,
\end{equation*}
where $p$ is a nonnegative integer as in \eqref{eq:ratApp} and
$\alpha\geq 0$.
\end{lemma}

\begin{proof}
The statement of the lemma follows from \eqref{eq:raterrfrac} and
Lemma \ref{lem:holderCont}. Indeed, for $t\in \widehat{I}_j$, we have
\begin{equation*}
\tilde{E}_k(t)-E(t)
=\big(\tilde{E}_k(t_j)-E(t_j)\big)+\big(E(t_j)-E(t)\big)
=\big(E_k^j-E(t_j)\big)+\big(E(t_j)-E(t)\big).
\end{equation*}
Hence, with $I=[0,T]$ and $\cI=\{0,1,\ldots,N\}$,
\begin{equation*}
\begin{aligned}
&\sup_{t\in I}\|\tilde{E}_k(t)-E(t)\|_{\mathcal{B}(\HH^\alpha,\HH)}\\
&\qquad
\leq\sup_{j\in\cI}\big(\|E_k^j-E(t_j)\|_{\mathcal{B}(\HH^\alpha,\HH)}
+ \sup_{t\in \widehat{I}_j}\|E(t_j)-E(t)\|_{\mathcal{B}(\HH^\alpha,\HH)}\big)\\
&\qquad \leq C(T)(k^{\min(\alpha\frac{p}{p+1},p)}+k^{\min(\alpha,1)})
\leq C(T)k^{\min(\alpha\frac{p}{p+1},1)},~k\le 1.
\end{aligned}
\end{equation*}
Finally, for $k>1$, the statement follows by stability.
\end{proof}

We are now ready to prove a bound for the weak error of the pure
time-discretization via \eqref{eq:spdeRA} of the stochastic wave
equation.

\begin{theorem}\label{thm:ratBound}
  Assume that $\norm{\Tr}{\Lambda^{\beta-1/2}Q\Lambda^{-1/2}}<\infty$
  and $X_0\in L_1(\Omega,\HH^{2\beta})$ for some $\beta\geq 0$ and
  $G\in \Cbb(\cH,\R)$. Then the weak error of the rational
  approximation algorithm \eqref{eq:spdeRA} of the stochastic wave
  equation described above is bounded by
\begin{equation}
\left|\bE \left(G(X_k^N)-G(X(T))\right)\right|
\leq Ck^{\min(2\beta\frac{p}{p+1},1)},~k>0.
\end{equation}
\end{theorem}
\begin{proof}
  We may use Theorem \ref{thm:waveGen} with $\T{E}(t)=\T{E_k}(t)$,
  $\T{P}=I$, $\VV=\HH$ and $L=I$, because
  $\norm{\Tr}{\Lambda^{\beta-1/2}Q\Lambda^{-1/2}}<\infty$ implies
  \eqref{qtdiscrete} by Corollary \ref{cor:Qassump} and Lemma
  \ref{lem:EkBound}. Thus, our claim follows by applying Lemma
  \ref{lem:supE} with $\alpha=2\beta$ to \eqref{eq:weakErrGen}.
\end{proof}
%
%
%%%%%%%%%%%%%%%%%%%%%%%%%%%%%%%%%%%%%%%%%%%%%%%%%%%%%%%%%%%%%%%%%%%%%%%%%%%%%%%
%
%
\subsection{Weak convergence of fully discrete schemes}\label{timespace}
In this section we will present an error estimate for a fully discrete
scheme. We will borrow the setting from \cite{BakBram}, where estimates
for the deterministic wave equation are proved. The spatial
discretization is performed by a standard continuous finite element
method and the time discretization, as above, by $I$-stable rational
approximations of the exponential function. We briefly describe this
method and state the error estimates from \cite{BakBram}.

We assume that $\cD$ is a convex polygonal domain and we let
$\{S_{h}^r\}_{0<h\leq 1}$, $r=2,3$, be a standard family of finite
element function spaces consisting of continuous piecewise polynomials
of degree $r-1$ with respect to a regular family of triangulations of
$\cD$.  Moreover, we define $S_{h,0}^r=\{v\in
S_h^r:v|_{\partial\cD}=0\}$, so that $S_{h,0}^r\subset \dH^1$.  With
$H^\beta$ denoting the standard Sobolev space we then have the error
estimate
\begin{equation}  \label{eq:ritzerror}
\|R_h v-v\| \leq Ch^\beta\|v\|_{H^\beta},
\quad v\in \dot{H}^1\cap H^{\beta},~\beta\in [1,r],
\end{equation}
where the Ritz projection
$R_h\colon\dot{H}^1\to S_{h,0}^r$ is defined by
\begin{equation*}
\la \nabla R_h v,\nabla \chi \ra=\la \nabla v,\nabla \chi \ra,\quad
\forall v\in \dH^1,\chi\in  S_{h,0}^r.
\end{equation*}
Further, we define the discrete Laplacian $\Lambda_h\colon
S_{h,0}^r\to S_{h,0}^r$ by
\begin{equation*}
(\Lambda_h\eta,\chi)=(\nabla\eta,\nabla\chi),
\quad\forall \eta,\chi\in S_{h,0}^r.
\end{equation*}
The homogeneous spatially semidiscrete wave equation is to
find $$u_h(t):=[u_{h,1}(t),u_{h,2}(t)]^T\in S_{h,0}^r\times S_{h,0}^r$$ such that
\begin{equation}\label{eq:femWave1}
\begin{bmatrix}\dot{u}_{h,1}\\\dot{u}_{h,2}\end{bmatrix}
+
\begin{bmatrix}0&-I\\ \Lambda_h& 0 \end{bmatrix}
\begin{bmatrix}u_{h,1}(t)\\u_{h,2}(t) \end{bmatrix}
=
\begin{bmatrix}0\\0\end{bmatrix}, \ t>0;
\quad
\begin{bmatrix}
 u_{h,1}(0)\\u_{h,2}(0)
\end{bmatrix}
=
\begin{bmatrix}
  P_{h,1}u_{0,1}\\P_{h,2}u_{0,2}
\end{bmatrix}
\end{equation}
Here $P_{h,1}\colon\dH^{0}\to S_{h,0}^r$ and $P_{h,2}\colon\dH^{-1}\to
S_{h,0}^r$ are the orthogonal projectors defined by $\langle
P_{h,i}f,\chi\rangle=\langle f,\chi\rangle$, $\forall \chi\in
S_{h,0}^r$, for $f\in\dH^0$ if $i=1$ and $f\in\dH^{-1}$ if $i=2$.

It is well known that $\Lambda_h$ has eigenpairs
$\{(\phi_{h,j},\lambda_{h,j})\}_{j=1}^{M_h}$, where
$\{\lambda_{h,j}\}_{j=1}^{M_h}$ is a positive, nondecreasing sequence
and $\{\phi_{h,j}\}_{j=1}^{M_h}$ an $\dH^0$-orthonormal basis of $S_{h,0}^r$. If
we write
\begin{equation*}
A_h:=\begin{bmatrix}0&-I\\ \Lambda_h& 0 \end{bmatrix}
\end{equation*}
and if $P_h=[P_{h,1},P_{h,2}]^T$ and $u_0:=[u_{0,1},v_{0,2}]^T$, then
\eqref{eq:femWave1} may be written
\begin{equation}\label{eq:femWave2}
\dot{u}_h+A_hu_h=0,\ t>0;\quad u_h(0)=P_hu_0.
\end{equation}
The operator $-A_h$ is the infinitesimal generator of a strongly
continuous semigroup $E_h(t)$ and the solution of \eqref{eq:femWave2}
is given by
\begin{equation*}
u_h(t)=E_h(t)P_hu_0.
\end{equation*}
Similarly to \eqref{eq:exgroup} the operator $E_h(t)$ has a
representation in terms of sine and cosine operators; i.e.,
\begin{equation*}
E_h(t)=\begin{bmatrix} C_h(t)&\Lambda_h^{-1/2}S_h(t)\\
-\Lambda_h^{1/2}S_h(t)&C_h(t)\end{bmatrix}
\end{equation*}
with $S_h(t)=\sin(t\Lambda_h^{1/2})$ and $C_h(t)=\cos(t\Lambda_h^{1/2})$.

The time discretization, as in the previous subsection, is performed by
$I$-stable rational single step schemes; i.e., schemes where the
rational function $R$ fulfills \eqref{eq:ratApp} for some positive
integer $p$.  The fully discrete problem on the same uniform grid as
in Subsection \ref{sec:semitime} then reads
\begin{equation}\label{eq:fdwave}
v_{h,k}^n=R(kA_h)v_{h,k}^{n-1},\ n=1,\ldots,N; \quad
 v_{h,k}^0=P_hu_0.
\end{equation}
We will henceforth write $E_{h,k}=R(kA_h)$ and the solution of
\eqref{eq:fdwave} may then be written as
\begin{equation}
v_{h,k}^n=E_{h,k}^nP_hu_0.
\end{equation}
The error estimate proved in \cite{BakBram} is as follows.  It
provides only a bound for the first component in $u$, which we express
by means of a projector $P^1$.

\begin{theorem}\label{thm:BakBra}
If $P^1\colon\cH\to\dH^{0}$ is defined as
$P^1x=x_1$ for $x=[x_1,x_2]^T\in\cH$, then
\begin{equation*}
\|P^1(E_{h,k}^nP_h-E(t_n))u_0 \|_{\dH^{0}}
\leq C(t_n)\big(h^r\|u_0 \|_{\HH^{r+1}}+k^p\|u_0\|_{\HH^{p+1}}\big),~t_n=nk\ge0.
\end{equation*}
\end{theorem}

Using the stability of $E(t)$ and $E^n_{h,k}$ and a standard
interpolation argument, this results in the following bound on the
error operator.

\begin{corollary}\label{cor:BakBra}
Under the assumptions of Theorem \ref{thm:BakBra} we have, for
$\beta \geq 0$,
\begin{equation*}
\|P^1(E_{h,k}^nP_h-E(t_n))\|_{\mathcal{B}(\HH^{\beta},\dH^0)}
\leq C(t_n)\big(h^{\min(\beta\frac{r}{r+1},r)}
+k^{\min(\beta\frac{p}{p+1},p)}\big),~t_n=nk\ge0.
\end{equation*}
\end{corollary}

We return to the stochastic wave equation whose fully discrete version
now reads, with $B_h:=P_hB=[0,P_{h,2}]^T$,
\begin{equation}\label{eq:fds}
X_{h,k}^j=E_{h,k}(X_{h,k}^{j-1}+B_h\Delta W^j),\ j=1,\ldots, N; \quad X_{h,k}^0=P_hX_0.
\end{equation}
The solution is given by
\begin{equation}\label{eq:xnh}
X^n_{h,k}=E_{h,k}^nP_hX_0+\sum_{j=1}^{n}E_{h,k}^{n-j+1}B_h\Delta W^{j}.
\end{equation}
As in the previous section we multiply by
$E_{h,k}^{N-n}$ and arrive at the drift free version
\begin{equation*}
Y^n_{h,k}=E_{h,k}^NP_hX_0+\sum_{j=1}^{n}E_{h,k}^{N-j+1}B_hW^{j}
\end{equation*}
and with piecewice constant interpolation
\begin{equation}\label{eq:Yhktilde}
\T{Y}_{h,k}(t)=\T{E}_{h,k}(T)P_hX_0+\int_0^t\T{E}_{h,k}(T-s)B_h\,\dd W(s)
\end{equation}
in exact analogy with the temporally semidiscrete case in
\eqref{eq:Yktilde}.

Next we bound the weak error for fully discrete schemes given by
\eqref{eq:fds}.  We only prove a result for the first component in
$X$.

\begin{theorem}\label{thm:fullWeakWave}
  Assume that $\norm{\Tr}{\Lambda^{\beta-1/2}Q\Lambda^{-1/2}}<\infty$
  and $X_0\in L_1(\Omega,\HH^{2\beta})$ for some $\beta\geq 0$. If
  $X_{h,k}^n$ is given by \eqref{eq:fdwave} and $X(t)$ is the weak
  solution \eqref{eq:X} of \eqref{spde} with $A,B$ as in
  \eqref{eq:waveDef}, then for $g\in
  C^2_\mathrm{b}(\dH^0,\mathbb{R})$,
  we have
\begin{equation*}
\big|\bE \big(g(X_{h,k,1}^N)-g(X_1(T))\big)\big|
\leq C(T)\big(h^{\min(2\beta\frac{r}{r+1},r)}
+k^{\min(2\beta\frac{p}{p+1},1)}\big).
\end{equation*}
\end{theorem}

\begin{proof}
  The function in \eqref{eq:Yhktilde} is clearly of the form
  \eqref{eq:tildeY2} with $\T{Y}_{h,k}(T)=X_{h,k}^N$ and we have
  already seen that
  $\norm{\Tr}{\Lambda^{\beta-1/2}Q\Lambda^{-1/2}}<\infty$ implies
  \eqref{qt}. Furthermore, $$\Tr(\T{E}_{h,k}(t)B_hQB_h^*\T{E}_{h,k}(t)^*)\le
  \Tr(B_hQB_h)<\infty,$$ as $B_hQB_h$ is a bounded operator with
  finite-dimensional range and hence it is of trace class. Therefore
  also \eqref{qtdiscrete} holds and Theorem \ref{thm:waveGen} can be
  applied with $\VV=\dH^0$, $L=P^1$ (as defined in Theorem
  \ref{thm:BakBra}), $\T{E}(t)=\T{E}_{h,k}(t)$, $\T{B}=B_h$ and
  $\T{P}=P_h$.  From Corollary \ref{cor:BakBra} and Lemma
  \ref{lem:holderCont}, as in the proof of Lemma \ref{lem:supE}, it
  follows that
\begin{equation*}
\sup_{t\in[0,T]}\|P^1(\T{E}_{h,k}(t)P_h-E(t))\|_{\cB(\HH^{2\beta},\dH^0)}\leq
C(T)\left( h^{\min(2\beta\frac{r}{r+1},r)}+k^{\min(2\beta\frac{p}{p+1},1)}\right).
\end{equation*}
Finally the statement of the theorem follows from inserting this into
\eqref{eq:weakErrGen}.
\end{proof}
%%%%%%%%%%%%%%%%%%%%%%%%%%%%%%%%%%%%%%%%%%%%%%%%%%%%%%%%%%%%%%%%%%%%%%%%%%%%%%
%
%
%%%%%%%%%%%%%%%%%%%%%%%%%%%%%%%%%%%%%%%%%%%%%%%%%%%%%%%%%%%%%%%%%%%%%%%%%%%%%%
\subsection{Strong convergence of fully discrete schemes}\label{strong}

It is a general phenomenon that the order of weak convergence is twice
the strong order under the same regularity of the noise. This
essentially turns out to be the case also for the stochastic wave
equation discretized by the method described in the previous
section. As we are not aware of any results on strong convergence of a
fully discrete approximation of the stochastic wave equation using
finite elements in the spatial domain, we give a short derivation of
the strong order. We remark that the strong convergence is studied for
a spatially semidiscrete finite element method in \cite{KLS}, for a
fully discrete leap-frog scheme in one spatial dimension in
\cite{Walsh}, and for a spatially semidiscrete scheme in one dimension
in \cite{San}.

First we form the strong error by taking the difference of
\eqref{eq:xnh} and \eqref{eq:X}, projecting onto the first component,
and taking norms:
\begin{equation*}
\begin{split}
&\bE\Big(\|P^1(X_{h,k}^N-X(T))\|_{\dH^0}^2\Big)
\leq C\bE\Big(\|P^1(E^N_{h,k}P_h-E(T))X_0\|_{\dH^0}^2\Big)\\
&\qquad
+C\bE\Big(\Big\|P^1\int_0^T(\T{E}_{h,k}(T-s)P_h-E(T-s))B\,\dd W(s)
\Big\|_{\dH^0}^2 \Big)
=:I_1+I_2.
\end{split}
\end{equation*}
If $X_0\in L_2(\Omega,\HH^\beta)$, then
\begin{align*}
I_1&\leq
C\|P^1(E_{h,k}P_h-E(T))\|^2_{\cB(\HH^\beta,\dH^0)}
\bE\big(\|X_0\|^2_{\HH^\beta}\big)
\\
& \leq
C(T)\big(h^{\min(\beta\frac{r}{r+1},r)}+k^{\min(\beta\frac{p}{p+1},p)}\big)^2
\|X_0\|_{L_2(\Omega,\HH^\beta)}^2
\end{align*}
by Corollary \ref{cor:BakBra}. For $I_2$ we use It\^o's isometry
\eqref{eq:Itoiso} to get
\begin{equation*}
\begin{split}
I_2&=\bE\Big(\Big\|\int_0^TP^1(\T{E}_{h,k}(T-s)P_h-E(T-s))B\,\dd W(s) \Big\|_{\dH^0}^2\Big) \\
&=\int_0^T\|P^1(\T{E}_{h,k}(T-s)P_h-E(T-s))BQ^{1/2} \|_\HS^2\,\dd s \\
&=\int_0^T\|P^1(\T{E}_{h,k}(T-s)P_h-E(T-s))B\Lambda^{\frac{1-\beta}{2}}\Lambda^{\frac{\beta-1}{2}}Q^{1/2} \|_\HS^2\,\dd s \\
&\leq \int_0^T\|P^1(\T{E}_{h,k}(T-s)P_h-E(T-s))\|^2_{\cB(\HH^\beta,\dH^0)}\|\Lambda^{\frac{\beta-1}{2}}Q^{1/2} \|_\HS^2\,\dd s\\
&\leq T\sup_{t\in[0,T]}\big(\|P^1(\T{E}_{h,k}(t)P_h-E(t))\|^2_{\cB(\HH^\beta,\dH^0)}\big)\|\Lambda^{\frac{\beta-1}{2}}Q^{1/2} \|_\HS^2\\
&\leq C(T)\|\Lambda^{\frac{\beta-1}{2}}Q^{1/2} \|_\HS^2
\big(h^{\min(\beta\frac{r}{r+1},r)}+k^{\min(\beta\frac{p}{p+1},1)}\big)^2,
\end{split}
\end{equation*}
where the first inequality follows from the fact that
$\|B\Lambda^{\frac{1-\beta}{2}}\|_{\cB(\dH^0,\cH^\beta)}=1$ combined with \eqref{eq:TrHS2}, and the last inequality
from Corollary \ref{cor:BakBra} and Lemma \ref{lem:holderCont} as in
the proof of Lemma \ref{lem:supE}. Combining the bounds for $I_1$ and
$I_2$ and taking square roots, we have shown the following result.

\begin{theorem}\label{thm:str}
  Let $\|\Lambda^{\frac{\beta-1}{2}}Q^{1/2} \|_\HS^2<\infty$ and
  $X_0\in L_2(\Omega,\HH^\beta)$ for some $\beta \ge 0$. Then the
  strong error of the approximation $X^N_{h,k,1}=P^1X^N_{h,k}$ of the
  displacement $X_1(T)=P^1X(T)$ in the stochastic wave equation is
  bounded by
\begin{equation*}
\|X_{h,k,1}^N-X_1(T)\|_{L_2(\Omega,\dH^0)}
\leq C(T)\big(h^{\min(\beta\frac{r}{r+1},r)}+k^{\min(\beta\frac{p}{p+1},1)}\big).
\end{equation*}
\end{theorem}

The regularity assumption on $Q$ in Theorem \ref{thm:fullWeakWave} implies the
assumption in Theorem \ref{thm:str}, see Theorem \ref{thm:aq} with
$s=\beta-1$ in \eqref{c2}.  Thus the claim that the weak rate is
essentially twice the strong rate is justified (if $\beta$ is not too
large) by comparing Theorems \ref{thm:fullWeakWave} and \ref{thm:str}.
Note also that the mean-square regularity is of order $\beta$ according to Theorem \ref{cor:QassumpA}.

\section{Application to parabolic equations}\label{sec:para}

Here, we give a detailed weak error analysis of a fully discrete
scheme for the linearized Cahn-Hilliard-Cook (CHC) equation and also
comment on the linear stochastic heat equation.

The linearized CHC equation, see \cite{LM}, is
\begin{equation*}
  \dd X + \Lambda^2 X\,\dd t=\dd W,\ t>0;\quad X(0)=X_0,
\end{equation*}
where now $\Lambda=-\Delta $ is the Laplacian together with
homogeneous Neumann boundary conditions.  To write the CHC equation in
the form \eqref{spde} we therefore set $\HH=\{f\in L_2(\cD):\la
f,1\ra=0\}$, $A=\Lambda^2$ with $D(\Lambda)=\{f\in H^2(\cD)\cap
\HH:\frac{\partial f}{\partial n}=0\}$, where $\cD$ is a convex
polygonal domain, and we take $\cU=\HH$ and $B=I$. We further define
the the spaces $\dH^\alpha=D(\Lambda^\alpha)$  in analogy with Section
\ref{sec:hyper}.  Thus, $\cH=\dH^0$, $D(A)=\dH^4$ and $-A$ is known to be the
infinitesimal generator of the analytic semigroup
$E(t)=\ee^{-tA}=\ee^{-t\Lambda^2}$ on $\HH$.

We recall the finite element spaces $S_{h}^r$ (without boundary
conditions) of order $r=2,3$ from Subsection \ref{timespace} and set
$\dot{S}_h^r=\{v\in S_h^r: \la v,1 \ra=0\}$.  We now define the
discrete Laplacian $\Lambda_h\colon \dot{S}_h^r\to \dot{S}_h^r$ and
the Ritz projector $R_h\colon\dot{H}^1\to \dot{S}_h^r$ in the
analogous way and we have an error bound of the same form as in
\eqref{eq:ritzerror}.  We set $A_h=\Lambda_h^2$ and note that $-A_h$
is the generator of an analytic semigroup $E_h(t)$ on
$\dot{S}_h^r$.  We consider only the backward Euler time-stepping and
therefore introduce $E_{h,k}=(1+kA_h)^{-1}$ and define
$\tilde{E}_{h,k}(t)$ in an analogous fashion to the case of the
wave equation, see \eqref{eq:Yktilde}.

We need error bounds for the approximation of the semigroup.
We claim that, for all $v\in \cH$,
\begin{equation}\label{eq:CHerror}
\|(E_{h,k}^nP_h-E(t_n))v\|
\leq C(h^{\alpha}+k^{\alpha/4})t_n^{-\alpha/4}\|v\|,
\quad t_n=kn,~\alpha\in [0,r],
\end{equation}
where $P_h\colon\cH\to \dot{S}_h^r$ denotes the $L_2(\cD)$-orthogonal
projection to $\dot{S}_h^r$. To see this we write
\begin{equation*}
E_{h,k}^nP_hv-E(t_n)v
=\big(E_{h,k}^nP_hv-E_h(t_n)P_hv\big)+\big(E_h(t_n)P_hv-E(t_n)v\big).
\end{equation*}
It is well known and follows by a simple spectral argument, as $A_h$ is self-adjoint positive semidefinite on $\dot{S}_h^r$, that the estimate
\begin{equation}\label{eq:CHrat}
\|E_{h,k}^nP_hv-E_h(t_n)P_hv\|
\leq C k^\gamma t_n^{-\gamma}\|v \|,\quad \gamma\in[0,1],
\end{equation}
holds for the backward Euler method \cite{LR}. It follows from the stability of the
finite element approximation and \cite[Corollary 5.3]{EL} that
\begin{equation}\label{eq:CHfem}
\| E_h(t)P_hv-E(t)v\|\leq Ch^{\gamma}t^{-\gamma/4}\|v \|,\quad\gamma\in[0,r].
\end{equation}
Thus, with $\gamma=\alpha/4\le r/4\le1$ in \eqref{eq:CHrat} and
$\gamma=\alpha$ in \eqref{eq:CHfem}, the estimate \eqref{eq:CHerror} follows.

It is also well known (see, for example, \cite[Theorem 6.13]{Pazy}) that
\begin{equation}\label{eq:eee}
\|(E(t)-E(s))A^{-\gamma}v\|\leq |t-s|^\gamma\|v\|,\quad \gamma\in [0,1],
\end{equation}
and therefore, taking also \eqref{eq:CHerror} into account, it follows that
\begin{equation}\label{eq:CHtildeError}
\|(\tilde{E}_{h,k}(t)P_h-E(t))v\|
\leq C(h^\alpha+k^{\alpha/4})t^{-\alpha/4}\|v\|, \quad \alpha\in [0,r].
\end{equation}
Indeed, for $t\in(t_{j-1},t_j]$ we have that
\begin{equation*}
\begin{split}
\|(\tilde{E}(t)P_h-E(t))v \|
&= \|(E_{h,k}^jP_h-E(t))v\|\\
&\leq
\|(E_{h,k}^jP_h-E(t_{j}))v\|+\|(E(t_j)-E(t))v\|.
\end{split}
\end{equation*}
For the first term \eqref{eq:CHerror} applies and for the second
term we use \eqref{eq:eee}:
\begin{equation*}
  \begin{split}
\|(E(t_j)-E(t))v\|&=\|A^{\alpha/4}E(t)(E(t_j-t)-I)A^{-\alpha/4}v\|\\
&\leq
\|A^{\alpha/4}E(t)\| \|(E(t_j-t)-I)A^{-\alpha/4}v\|
\leq
C k^{\alpha/4} t^{-\alpha/4}\|v\|.
\end{split}
\end{equation*}
Finally, we recall the smoothing property of the backward Euler
scheme. It follows from \cite[Lemma 7.3]{Thomeebook} by stability and
interpolation that for $t\in (t_{j-1},t_j]$,
\begin{equation*}
\|A_h^\alpha\tilde{E}_{h,k}(t)P_hv \|
=\|A^{\alpha}_hE_{h,k}^jP_hv\|
\leq Ct_j^{-\alpha}\|P_hv\|\leq Ct^{-\alpha}\|v\|, \quad\alpha\ge 0.
\end{equation*}
Therefore,
\begin{equation}\label{eq:AEhkbound}
\|A_h^\alpha\tilde{E}_{h,k}(t)P_hv \|\leq Ct^{-\alpha}\|v\|, \quad\alpha\ge 0,~t>0.
\end{equation}
We are now in the position to prove the following
estimate for the weak error in case of the linearized CHC equation. As
it was the case for the wave equation, the weak convergence rate is
twice that of the strong convergence rate \cite{LM} (up to a
logarithmic factor) under essentially the same regularity requirements
on $A$ and $Q$.

\begin{theorem}
  Let $X$ be the solution of \eqref{spde} and $X^n_{h,k}$ be given by
  \eqref{eq:xnh} with spaces and operators described above and
  $B_h=P_h$. Assume $G\in\Cbb(\HH,\R)$, $X_0\in L_1(\Omega,\HH)$ and
\begin{equation}\label{eq:CHCass1}
\|A^{(\beta-2)/2}Q\|_{\Tr}\leq K,
\quad
\|A_h^{(\beta-2)/2}P_hQ\|_{\Tr}\leq K,
\end{equation}
for some  $\beta\in(0,\tfrac{r}2]$ and $K>0$. Then there is $C$
depending on $T$, $K$, $\|X_0\|_{L_1(\Omega,\cH)}$, and
$\|G\|_{\Cbb(\HH,\R)}$ such that, for
$Nk=T,~h^4+k<T$,
\begin{equation}\label{eq:c}
\big|\EE(G(X^N_{h,k})-G(X(T)))\big|
\leq C \big(h^{2\beta}+k^{\beta/2}\big)\log(\tfrac{T}{h^4+k}).
\end{equation}
\end{theorem}
\begin{proof}
  Assumption \eqref{eq:CHCass1} guarantees that
  $\|A^{\frac{\beta-2}{4}}Q^{\frac12}\|_{\HS}<\infty$ in view of
  Theorem \ref{thm:aq}.  This in its turn implies that $X$ exists,
  as shown in \cite{LM}.  Will use Theorem \ref{main} with $\cO$ as
  in \eqref{eq:F1}. Furthermore, let
  $\tilde{B}=P_h,\tilde{X}_0=P_h X_0$,
  $\tilde{E}(t)=\tilde{E}_{h,k}(t)$ and $\tilde{Y}_{h,k} (t)$ be defined as
  in \eqref{eq:Yhktilde}, whence $\tilde{Y}_{h,k}(T)=X^N_{h,k}$. The use of
  Theorem \eqref{main} is justified since
  $\tilde{E}_{h,k}(t)P_hQ[\tilde{E}_{h,k}(t)P_h]^*$ is bounded and of
  finite rank so that \eqref{qtdiscrete} holds.

  We write $\tilde{F}_{h,k}(t)=\tilde{E}_{h,k}(t)P_h -E(t)$ and recall
  \eqref{eq:CHtildeError}.  For the first term in \eqref{e1} we use
  that $\tilde{Y}_{h,k}(0)-Y(0)=\tilde{F}_{h,k}(T)$,
  \eqref{eq:CHtildeError}, and the bound for $u_x$ in
  \eqref{eq:uderBound} to get
\begin{equation}\label{eq:firstterm}
\begin{aligned}
&\Big|\bE\int_0^1\Big\la u_x\big(Y(0)+s(\tilde{Y}_{h,k}(0)-Y(0) ),0\big),
\tilde{Y}_{h,k}(0)-Y(0) \big\ra\,\dd s\Big|
\\ &\qquad
\leq \sup_{x\in\HH}\|u_x(x,0)\| \, \bE\big(\|\tilde{F}_{h,k}(T) X_0\|\big)
\\ &\qquad
\leq \|G\|_{\Cbb(\HH,\R)} C (h^{2\beta}+k^{\beta/2})
T^{-\beta/2}\EE\big(\|X_0\|_{\cH}\big).
\end{aligned}
\end{equation}

For the second term of \eqref{e1} we have by \eqref{eq:trace1} and
repeated use of \eqref{eq:TrHS2} that
\begin{equation*}
\begin{split}
&\Big|\EE\int_0^T\Tr\Big(u_{xx}\big(\tilde{Y}_{h,k}(t),t\big)
(\tilde{E}_{h,k}(T-t)P_h+E(T-t))Q\tilde{F}_{h,k}(T-t)^*\Big)\, \dd t\Big|\\
&\quad \leq
\EE\Big(\int_0^T\|u_{xx}(\tilde{Y}_{h,k}(t),t)
(\tilde{E}_{h,k}(T-t)P_h+E(T-t))Q\|_{\Tr}
\|\tilde{F}_{h,k}(T-t) \|_{\cB(\HH)}\, \dd t \Big)\\
&\quad \leq \sup_{(x,t)\in \HH\times [0,T]}\|u_{xx}(x,t)\|_{\cB(\HH)}\\
&\quad\quad\times\int_0^T
\|(A_h^{-(\beta-2)/2}\tilde{E}_{h,k}(t) A_h^{(\beta-2)/2}P_h
+A^{-(\beta-2)/2} E(t) A^{(\beta-2)/2})Q\|_{\Tr}\\
&\quad\quad \quad\times
\|\tilde{F}_{h,k}(t)\|_{\cB(\HH)}\, \dd t \\
&\quad \leq \|G\|_{\Cbb(\HH,\R)}
\int_0^T\Big(\|A_h^{-(\beta-2)/2}\tilde{E}_{h,k}(t)\|_{\cB(\HH)}\,
\|A_h^{(\beta-2)/2} P_hQ\|_{\Tr}  \\
&\quad\quad +\|A^{-(\beta-2)/2} E(t)\|_{\cB(\HH)} \,
\|A^{(\beta-2)/2}Q\|_{\Tr}\Big)\,\|\tilde{F}_{h,k}(t)\|_{\cB(\HH)} \, \dd t \\
&\quad \leq \|G\|_{\Cbb(\HH,\R)}
\Big(\| A_h^{(\beta-2)/2}P_hQ\|_{\Tr}
+\|A^{(\beta-2)/2}Q\|_{\Tr} \Big)\\
&\quad\quad \times \int_0^T\Big(
 \| A_h^{-(\beta-2)/2}\tilde{E}_{h,k}(t)\|_{\cB(\HH)}
+\|A^{-(\beta-2)/2} E(t)\|_{\cB(\HH)}\Big)
\|\tilde{F}_{h,k}(t)\|_{\cB(\HH)}\, \dd t.
\end{split}
\end{equation*}
By \eqref{eq:CHCass1} the factors in front of the integral are bounded
by $2K \|G\|_{\Cbb(\HH,\R)}$.

We proceed by splitting the integral in two as
$\int_0^T=\int_0^{h^4+k}+\int_{h^4+k}^T$. For the first integral we
notice that the last factor of the integrand is uniformly bounded and
hence, by the analyticity of $E(t)$ and \eqref{eq:AEhkbound} with $\alpha=-(\beta-2)/2$,
\begin{equation*}
\begin{split}
&\int_0^{h^4+k}\Big(\|A_h^{-(\beta-2)/2}\tilde{E}_{h,k}(t)\|_{\cB(\HH)}
+\|A^{-(\beta-2)/2} E(t)\|_{\cB(\HH)}\Big)
\|\tilde{F}_{h,k}(t)\|_{\cB(\HH)}\, \dd t\\
&\qquad \leq C\int_0^{h^4+k} t^{(\beta-2)/2} \,\dd t
=C(h^4+k)^{\beta/2}\le C(h^{2\beta}+k^{\beta/2}).
\end{split}
\end{equation*}
For the second part we use again the analyticity of $E(t)$,
\eqref{eq:CHtildeError} with $\alpha=2\beta$ and \eqref{eq:AEhkbound} to get
\begin{equation*}
\begin{aligned}
&\int_{h^4+k}^T \Big(\|A_h^{-(\beta-2)/2}\tilde{E}_{h,k}(t)\|_{\cB(\HH)}
+\|A^{-(\beta-2)/2} E(t)\|_{\cB(\HH)}\Big)
\|\tilde{F}_{h,k}(t)\|_{\cB(\HH)}\, \dd t\\
&\quad \leq C\int_{h^4+k}^T
\big(t^{(\beta-2)/2}+t^{(\beta-2)/2}\big)\big(h^{2\beta}+k^{\beta/2}\big)t^{-\beta/2}
\, \dd t\\
&\quad =C(h^{2\beta}+k^{\beta/2})\int_{h^4+k}^Tt^{-1}\,\dd t
= C \log(\tfrac{T}{h^4+k})(h^{2\beta}+k^{\beta/2}).
\end{aligned}
\end{equation*}
\end{proof}
\begin{remark}
We refer to \cite[Theorem 4.4]{KLLweak} for $h$-independent conditions
guaranteeing \eqref{eq:CHCass1} and the remarks after its proof for further discussions of the abstract conditions. Furthermore, the dependence on $T$ of $C$ in \eqref{eq:c} can be removed if we assume that $X_0\in L_1(\Omega,\dH^{2\beta})$ by using the deterministic error estimate for smooth initial data from \cite{LM} in \eqref{eq:firstterm}.
\end{remark}

\begin{remark}\label{rem:dp}
  The weak convergence of the finite element space discretization and
  backward Euler time discretization of stochastic heat equation with
  additive noise was considered in \cite{debusscheprintems}. The
  results there can be recovered using the fully discrete
  deterministic estimates
$$\|E^n_{h,k}P_h-E(t_n)\|_{\cB(\cH)}
\le C(h^2+k)t_n^{-1}
$$
and
$$\|\Lambda^{\alpha}E^n_{h,k}P_h \|_{\cB(\cH)}  +  \|\Lambda^{\alpha}E(t_n) \|_{\cB(\cH)}
\le Ct_n^{-\alpha},\quad\alpha\in [0,\tfrac{1}{2}],
$$
together with Theorem \ref{main}. The technicalities are the same as in the spatially
semidiscrete case \cite[Theorem 4.1]{KLLweak} under the same symmetric
condition $$\|\Lambda^{\frac{\beta-1}{2}}Q^{\frac12}\|_{\HS}<\infty.$$
We do not detail this here any further as it recovers a known result,
only with perhaps a more transparent proof.
\end{remark}

%\newpage

\end{document}